\documentclass[11pt]{amsart}
\usepackage{polski}
\usepackage[T1]{fontenc}
\usepackage{amsmath, amsthm, amssymb, amsfonts, enumerate}
\usepackage{enumitem}
\usepackage[colorlinks=true,linkcolor=blue,urlcolor=blue]{hyperref}
\usepackage{dsfont}  
\usepackage{color}
\usepackage{geometry}
\usepackage{todonotes}
\usepackage{enumerate}

\newtheorem{theorem}{Theorem}[section]
\newtheorem{remark}[theorem]{Remark}
\newtheorem{assumption}[theorem]{Assumption}
\newtheorem{lemma}[theorem]{Lemma}
\newtheorem{proposition}[theorem]{Proposition}
\newtheorem{corollary}[theorem]{Corollary}
\newtheorem{definition}[theorem]{Definition}
\newtheorem{example}[theorem]{Example}
\theoremstyle{plain}

\newcommand{\field}[1]{\mathbb{#1}}

\newcommand{\R}{\field{R}}
\newcommand{\N}{\field{N}}

\newcommand{\E}{\field{E}}
\renewcommand{\P}{\field{P}}

\usepackage[utf8]{inputenc}
\usepackage[english]{babel}
\usepackage{indentfirst}
\usepackage{graphicx}
\usepackage{pdfpages}
\usepackage[title]{appendix}
\usepackage{fancyhdr}

\usepackage{enumitem}
\usepackage{amsmath}
\usepackage{amssymb} 
\usepackage{bm}
\usepackage{mathrsfs}
\usepackage{amsthm}
\usepackage{comment}
\usepackage{hyperref}
\usepackage{dsfont} 
 
\usepackage{mathtools}
\DeclareMathOperator*{\argmin}{arg\,min}
\DeclareMathOperator*{\argmax}{arg\,max}
\DeclareMathOperator*{\supp}{supp}

\DeclareMathOperator*{\esssup}{ess\,sup}

\title[A unifying framework for submodular mean field games]{A unifying framework for submodular mean field games} 
\author[Dianetti]{Jodi Dianetti}
\author[Ferrari]{Giorgio Ferrari}
\author[Fischer]{Markus Fischer}
\author[Nendel]{Max Nendel}

\date{\today}

\numberwithin{equation}{section}
\begin{document}
\maketitle 

\begin{abstract}
We provide an abstract framework for submodular mean field games and identify verifiable sufficient conditions that allow to prove existence and approximation of strong mean field equilibria in models where data may not be continuous with respect to the measure parameter and common noise is allowed. The setting is general enough to encompass qualitatively different problems, such as mean field games for discrete time finite space Markov chains, singularly controlled and reflected diffusions, and mean field games of optimal timing. Our analysis hinges on Tarski's fixed point theorem, along with technical results on lattices of flows of probability and sub-probability measures.\\
[0.5em]
[0.5em]
{\textbf{AMS(2020) subject classification}}: 49N80, 91A16, 93E20, 06B23.
\end{abstract}

\section{Introduction} 
\label{sec:intro}



Mean field games (MFGs in short) are limit models for non-cooperative symmetric $N$-
player games with interaction of mean field type as the number of players $N$ tends to infinity. They have been proposed independently by \cite{HuangMalhameCaines06} and \cite{LasryLions07}, and since their introduction they have attracted increasing interest in various fields of Mathematics ranging from PDE theory to stochastic analysis and game theory, as well as in applications in Economics, Finance, Biology, and Engineering, among others; we refer, for instance, to the recent two-volume book \cite{CarmonaDelarue18} for an extensive presentation of theoretical results and applications.

The interest in identifying a key property that allows to prove existence and approximation of equilibria for a general class of MFGs has motivated our study. 
Inspired by the early contribution of Topkis \cite{To} on submodular $N$-player games in a static setting, we identify submodularity as a relevant structural condition and explore the flexibility of lattice-theoretical techniques in MFGs enjoying a submodular structure. 
Submodular MFGs have already been considered in the literature; see \cite{Adlakhaetal} for a class of stationary discrete time games, \cite{Wiecek17} for a class of finite state MFGs with exit, \cite{CarmonaDelarueLacker16} for optimal timing MFGs, and \cite{dianetti.ferrari.fischer.nendel.2019} for MFGs involving a regularly controlled one-dimensional It\^o-diffusion. 
In this work, we push the analysis of ours \cite{dianetti.ferrari.fischer.nendel.2019} much forward, and we provide an abstract framework for submodular MFGs, which embeds qualitatively different problems and allows to show existence and approximation of their mean field equilibria. 
The results of our work can be informally presented as follows. 
\begin{enumerate}
    \item The submodular structure of the game yields an alternative way of establishing existence of MFG solutions by using the lattice-theoretical Tarski's fixed point theorem, rather than topological fixed point results. 
    This allows to treat systems with coefficients that are possibly discontinuous in the measure variable, as well as to  prove existence of strong solutions in settings involving a common noise.
    
    \item The set of MFG solutions enjoys a lattice structure so that there exist a minimal solution and a maximal solution with respect to a suitable order relation.      
    \item A learning procedure, which consists of iterating the best-response-map (thus computing a new flow of measures as best-response  to the previous measure flow) converges to the minimal (or the maximal) MFG solution, for appropriately chosen initial measure flows.   
\end{enumerate}  

These claims are made precise in Theorem \ref{theorem.main.general}, under suitable assumptions that are formulated at a general abstract level. 
Those requirements do not involve nondegeneracy of the underlying noise and are satisfied in a variety of formulations of the mean field game problem, including deterministic frameworks.
Clearly, the setting of our previous work \cite{dianetti.ferrari.fischer.nendel.2019} is included. Furthermore, in this paper we highlight the flexibility of the approach by considering four qualitatively different problems, in which the representative agent minimization problem involves as a state variable: 
(i) a finite state discrete time Markov chain (cf.\ Section \ref{SectFiniteMarkov}); (ii) a singularly controlled It\^o-diffusion, possibly affected by a common noise (cf.\ Section \ref{section singular controls}); 
(iii) an It\^o-diffusion facing a reflecting boundary condition (cf.\ Section \ref{section reflected diff}); 
a general progressive stochastic process whose evolution can be stopped by the representative player (cf.\ Section \ref{section OS}). Here a common source of noise is also allowed.
For each of these examples, existence and approximation results are derived through a suitable application of Theorem \ref{theorem.main.general}. It is worth noting that fine properties of lattices of probability and sub-probability measures are needed in order to apply Theorem \ref{theorem.main.general} in the different examples. As we were unable to find a precise reference for those properties, we present them in brevity in Section \ref{section lattices}. Given the generality of the setting in which they are obtained, we believe that those findings are of interest on their own and might be a useful technical tool in other works as well.

The approach that we follow in this paper focuses exclusively on the representative agent minimization problem, without reformulating the problem in terms of a related  forward-backward system or of the master equation. Whether those reformulations of the mean field game problem allow to obtain results of a similar fashion as ours is, to the best of our knowledge, an open question that we leave for future research.

\subsection{Existence and approximation results in MFGs}

Questions of existence and approximation of mean field equilibria have been addressed in the literature at various degrees of generality and through different mathematical techniques. 

General existence results for solutions to the MFG problem can be obtained through Banach's fixed point theorem if the time horizon is small (cf.\ \cite{HuangMalhameCaines06}). 
For arbitrary time horizon, a version of the Brouwer-Schauder fixed point theorem, including generalizations to multi-valued maps, can be used; cf.\ \cite{cardaliaguet2010} and \cite{Lacker15}  (see also \cite{Fu&Horst17} in the context of MFGs with singular controls). 
In the presence of a common noise (i.e., an aggregate source of randomness), the existence of a weak MFG solution (i.e., not adapted to the common noise) can be established for a general class of MFGs.  
On the other hand, the existence of a strong MFG solution (i.e., adapted to the common noise) has been addressed mainly under conditions which imply uniqueness of equilibria.  
For example, in \cite{CarmonaDelarueLacker16} an analogue of the famous result by Yamada and Watanabe is derived,
and it is used to prove existence and uniqueness of a strong solution under the Lasry-Lions monotonicity conditions (see \cite{LasryLions07}). 
Under lack of uniqueness, existence of strong solutions remains mainly an open question. 
 
Since uniqueness of equilibria in game theory is the exception rather than the rule, it is not surprising that multiple solutions often arise also in MFGs.
This phenomenon has been investigated mainly on a case by case basis, and specific examples with multiple solutions have been presented in the recent literature \cite{BardiFischer18, Cecchin&DaiPra&Fischer&Pelino19, DelarueFT19, tchuendom}, among others. 
Interestingly, the submodularity assumption appears implicitly in a number of classical linear-quadratic models (see, e.g., \cite{Bensoussan-etal-2016}) and in \cite{BardiFischer18, Campietal, Cecchin&DaiPra&Fischer&Pelino19, DelarueFT19}, although this property is not exploited therein. 
The increasing interest in non-uniqueness of solutions together with the perspective of characterizing many models through a unique key structural property has been one of the main motivations for our study of submodular MFGs.  

Once existence is established, it is natural to investigate how and whether solutions to MFGs can be approximated in a constructive way. 
This problem has been addressed by Cardaliaguet and Hadikhanloo \cite{CardaliaguetHadikhanloo17}.  
They analyze a learning procedure -- similar to what it is known as "fictitious play" (cf.\ \cite{HofbauerSandholm02} and the references therein) --  where the representative agent, starting from an arbitrary flow of measures, computes a new flow of measures by updating the average over past measure flows according to the best-response to that average.  
For potential mean field games, the authors establish convergence of this kind of fictitious play via PDE methods. 
Similar approaches have been further developed in some more recent works (see \cite{elie2019approximate, perrin2020fictitious, xie2020provable}, among others) with the help of machine learning techniques, providing a rich set of tools able to address computational aspects in MFGs.  
As already discussed, our result also contributes to the approximation question since the submodularity condition provides convergence of a simple learning procedure \`a la Topkis \cite{Topkis11}, consisting of iterating the best-response map (see \cite{dianetti.ferrari.fischer.nendel.2019}, and also \cite{DianettiFerrari} in the context of $N$-player games).
In particular, this type of algorithm seems to be quite promising when combined with reinforcement learning methods, as shown in the recent \cite{lee.Rengarajan.Kalathil.Shakkottai2021} for stationary discrete time finite state MFGs with complementarities.


\subsection{Examples.}

We now discuss in more detail the applications of Theorem \ref{theorem.main.general} that we present in this work, by also reviewing the related literature.

\subsubsection{Submodular mean field games with finite state discrete time Markov chains}

We start with a simple class of finite state discrete time MFGs where expected costs are to be minimized over a finite time horizon. Control acts on the Kolmogorov equation, that is, on the transition matrix that determines the evolution of the state probability vector. Mean field interaction only occurs through the measure variable appearing in the cost coefficients. We relate our model to the general set-up of Section~\ref{section general case} and provide sufficient conditions so that part~a) of Theorem~\ref{theorem.main.general} applies, yielding existence of solutions. We also give a simple example of a class of two-state models satisfying those conditions. They are related to the continuous time two-state MFGs studied in \cite{gomes.velho.wolfram.2013, Cecchin&DaiPra&Fischer&Pelino19}, also see \cite{bayraktar.zhang.2020}, which exhibit multiple solutions. 

The study of finite state discrete time MFGs goes back to \cite{gomes.mohr.souza.2010}, where existence and convergence to equilibrium for a class of finite horizon problems were established. For discrete MFGs of this type satisfying an analogue of the Lasry-Lions monotonicity condition, convergence of a ``fictitious play'' learning procedure is proved in \cite{hadikhanloo.silva.2019}. There, discrete models are also shown to approximate corresponding continuous time and space MFGs. Existence of solutions for a general class of finite and infinite horizon discrete MFGs is established in \cite{doncel.gast.gaujal.2019}, and their connection with the underlying $N$-player games investigated. Discrete time MFGs with more general state space have been studied recently under various optimality criteria; see \cite{saldi.basar.raginsky.2018, saldi.basar.raginsky.2020} for infinite horizon discounted cost and risk sensitive problems, respectively, \cite{Wiecek20} for ergodic MFGs, and \cite{bonnans.lavigne.pfeiffer.2021} for risk averse problems. Existence of solutions in those works is established through a topological fixed point theorem; in particular, cost coefficients are assumed to depend continuously on the measure variable. Although our simple discrete models fall under the framework of, for instance, \cite{doncel.gast.gaujal.2019}, the continuity assumptions there are not needed here, since here, as in the aforementioned \cite{Wiecek17} and \cite{Adlakhaetal}, we rely on an order-theoretic fixed point result. Lastly, we mention that our finite state MFGs do not involve common noise. Choosing a common noise for finite state problems is in fact less straightforward than in the usual continuous space setting; see the recent works \cite{bayraktar.cecchin.cohen.delarue.2020, bayraktar.cecchin.cohen.delarue.2021} for continuous time finite state problems.

\subsubsection{Submodular mean field games with singular controls}

The number of papers considering MFGs of singular stochastic controls is still relatively limited. \cite{Fu&Horst17} employs a relaxed approach in order to establish existence for a general class of MFGs involving singular controls, while the more recent \cite{Fu19} extends the analysis to MFGs in which interaction takes place both through states and controls. In \cite{Campietal} and \cite{Guo&Xu18} MFGs for finite-fuel follower problems are considered. By employing, respectively, the connection to problems of optimal stopping and PDE methods, the structure of the mean-field equilibrium as well as its connection to Nash equilibria for the corresponding $N$-player stochastic differential games is derived. Finally, \cite{CaoGuo} and \cite{dianettietal} study stationary MFGs, i.e.\ games in which the interaction comes through the stationary distribution of the population of players. \cite{dianettietal} considers ergodic and discounted performance criteria, and studies the relation across the corresponding equilibria; in \cite{CaoGuo} the representative player can employ two-sided controls in order to adjust a geometric dynamics and optimize a certain discounted payoff. It is worth noting that none of the previous contributions allows for the presence of common noise, which we can instead treat in our analysis. We can indeed show that the class of submodular MFGs with geometric dynamics that we consider in Section \ref{section geometric BM} admits strong equilibria (i.e.\ adapted to the common noise), which can in fact also be approximated through the previously discussed learning algorithm \`a la Topkis. In the case of a general nonconvex setting, a weak formulation of the singular control MFG is employed, and existence of mean field equilibria is proved by means of an approximation result through Lipschitz-continuous controls. Furthermore, convergence of the learning procedure is also established (see Section \ref{section nonconvex singular}).

\subsubsection{Submodular mean field games with reflecting boundary conditions}
 Theorem \ref{theorem.main.general} yields also existence and approximation of equilibria for submodular MFGs in which the representative player can employ regular controls in order to adjust the drift of a one-dimensional It\^o-diffusion which is constrained, via a Skorokhod reflection, to live in a bounded interval (cf.\ Section \ref{section reflected diff}). These models have received recent interest since they naturally arise as suitable limits of interacting queuing systems, see \cite{bayraktar.budhiraja.cohen.2019.AAP} and \cite{bayraktar.budhiraja.cohen.2018numerical}. As in \cite{bayraktar.budhiraja.cohen.2019.AAP}, we employ a weak (distributional) approach, and, by enforcing additional mild technical requirements on the data of the problem, an application of Tanaka's formula for continuous semimartingales allows to embed the considered MFG into the class of abstract submodular MFGs for which Theorem \ref{theorem.main.general} holds. Then, existence and approximation of mean field equilibria follow.

\subsubsection{Supermodular mean field games with optimal stopping}

In Section \ref{section OS} we consider a class of MFGs where the representative agent can choose a stopping time in order to stop the evolution of a general multi-dimensional progressive process, while maximizing a certain reward functional. The model is formulated by including the presence of a common noise. By assuming that the running profit function is increasing with respect to the stochastic order put on the lattice of sub-probability measures, the game enjoys a supermodular (rather than submodular, since here we are dealing with a maximization problem) structure that allows to invoke Theorem \ref{theorem.main.general} and show existence of equilibria. Furthermore, under suitable continuity requirements, convergence of a learning procedure is obtained.

Models involving MFGs of optimal stopping have been considered in the economic literature mostly in stationary settings (see \cite{Luttmer} and \cite{Miao} in the context of industry equilibria) and, more recently, under greater generality also in the mathematical literature; 
see \cite{bouveret.dumitrescu.tankov.20}, \cite{Tankov2}, and \cite{Bertucci}. 
Using a relaxed solution approach, in \cite{bouveret.dumitrescu.tankov.20} and \cite{Tankov2} an It\^o-diffusive setting  not allowing for a common noise is considered (see also Example \ref{os example MFGs Tankov} in Section \ref{section OS}). 
In \cite{Bertucci}, an analytical approach to MFGs of optimal stopping is developed through the study of the associated variational inequality. 
Explicit use of the supermodular property and of the Tarski's fixed point theorem in a MFG of stopping with common noise is made in  \cite{carmona.delarue.lacker.2017.timing} (see also Example \ref{os example MFGs timing} in Section \ref{section OS}). 

\subsection{Outline of the paper} The rest of the paper is organized as follows. Section \ref{section general case} presents the general approach to submodular MFGs. There, we state and prove Theorem \ref{theorem.main.general}. Section \ref{section lattices} derives the properties of lattices of probability and sub-probability measures needed in the paper. The remaining sections deal with applications of the abstract setup: Section \ref{SectFiniteMarkov} deals with MFGs having discrete time finite space Markov chains as state variables; Section \ref{section singular controls} considers MFGs with singular controls; Section \ref{section reflected diff} treats MFGs with reflecting boundary conditions, while MFGs of optimal stopping are addressed in the final Section \ref{section OS}. For the reader's convenience, we collect some lattice-theoretical preliminaries in Appendix \ref{append.lattice}.


\section*{General notation}
For a fixed finite time horizon $T \in (0,\infty)$, we introduce the following canonical spaces:
\begin{enumerate}
\item $\mathcal{C}$ denotes the  space of $\mathbb{R}$-valued  continuous functions defined on $[0,T]$, endowed with the supremum norm and the Borel $\sigma$-algebra $\mathcal{B}(\mathcal{C})$ generated by the supremum norm.
\item For a set $A \subset \mathbb R$, let $\Lambda$ denote the set of \emph{deterministic relaxed controls} on $[0,T] \times A$; that is, the set of positive measures $\lambda$ on $[0,T] \times A$ such that $\lambda([s,t] \times A)=t-s$ for all $s,t \in [0,T]$ with $s<t$. The set $\Lambda$ is endowed with the topology of weak convergence of probability measures, and $\mathcal B (\Lambda)$ denotes the related Borel $\sigma$-algebra.
\item $\mathcal{D}$ denotes the Skorokhod space of $\mathbb{R}$-valued c\`adl\`ag functions, defined on $[0,T]$, endowed with the Borel $\sigma$-algebra $\mathcal{B}(\mathcal{D})$ generated by the Skorokhod topology.  On the space $\mathcal{D}$ consider the \emph{pseudopath topology} $\tau_{pp}^{\text{\emph{\tiny T}}}$; that is, the topology on $\mathcal{D}$ induced by the convergence in the measure $dt + \delta_T$ on the interval $[0,T]$, where $dt$ denotes the Lebesgue measure, and $\delta_T$ denotes the Dirac measure at the terminal time $T$. For the topological space $(\mathcal{D} ,\tau_{pp}^{\text{\emph{\tiny T}}})$, the Borel $\sigma$-algebra induced by the topology $\tau_{pp}^{\text{\emph{\tiny T}}}$, coincides with the $\sigma$-algebra induced by the Skorokhod topology (see the Appendix in \cite{li&zitkovic2017}). 
\item $\mathcal{D}_{\uparrow}$ denotes the set of elements of $\mathcal D$ which are nonnegative and nondecreasing, endowed with the Borel $\sigma$-algebra $\mathcal{B}(\mathcal{D}_{\uparrow})$ induced by the Skorokhod topology. Note that $\mathcal{D}_{\uparrow}$ is a closed subset of the topological space $(\mathcal{D} ,\tau_{pp}^{\text{\emph{\tiny T}}})$.
\item $\mathcal V$ denotes the set of elements of $\mathcal D$ with bounded total variation, endowed with the Borel $\sigma$-algebra $\mathcal{B}(\mathcal V)$ induced by the Skorokhod topology. Furthermore, the space $\mathcal V$ is a closed subset of the topological space $(\mathcal{D} , \tau_{pp}^{\text{\emph{\tiny T}}})$. 
\end{enumerate}


\section{A general approach to submodular MFGs}
\label{section general case}
In this section, we consider an abstract version of a mean field game.
The aim of this section is to collect fundamental structural conditions and arguments, which provide a common basis for the examples treated in the next sections. For the lattice-theoretical notions and preliminaries that are used throughout this section and the rest of this paper, we refer to Appendix \ref{append.lattice}.

\subsection{Formulation of the abstract model}
Let $(L, \leq^{\text{\tiny{$L$}}})$ be a complete and Dedekind super complete lattice, which represents the set of possible distributions of players, see Definition \ref{def.dedekind}.
Let $E$ be the set of strategies of the representative player.
The set $E$ is endowed with a topology and a map $p: E\to L$, which can be interpreted as a projection, which maps each strategy to a related distribution. 
The representative player wants to minimize a cost functional 
$
J: E \times L \to \mathbb{R}, 
$
depending also on the distribution of her opponents.

We make the following assumption (see also Remark \ref{remark assumptions non topological} for a generalization):
\begin{assumption}
\label{assumption.bestresponsemap.nonempty}
For every $\mu \in L$, we assume that: 
\begin{enumerate}
    \item The set $\argmin_{E} J(\cdot, \mu)$ is nonempty and $J(\cdot,\mu)$ is lower semicontinuous;
    \item \label{assumption.bestresponsemap.nonempty.continouty projection} For any sequence  $(\nu^n)_n \subset  \argmin _E J(\cdot, \mu)$
such that $p (\nu^n)$ is nondecreasing or nonincreasing in $L$,  there exists a subsequence $(n_j)_{j \in \N}$ and $\nu \in \argmin_{E} J(\cdot, \mu)$ such that $\nu^{n_j}$ converges to $\nu$ as $j \to \infty$ and $ p\nu = \sup_j p(\nu^{n_j})$ or $p\nu = \inf_j p(\nu^{n_j})$, respectively.
\end{enumerate} 
\end{assumption} 

For $\mu\in L$, we define the set of best responses $R(\mu)\subset L$ by
\[
 R(\mu):=p\bigg(\argmin_{\nu\in E} J(\nu, \mu)\bigg).
\]

\begin{definition}
\label{definition.MFG.solution}
 We say that $\mu\in L$ is a \textit{mean field game equilibrium} if $\mu\in R(\mu)$, i.e. $\mu$ is a fixed point of the best-response-map.
\end{definition}

\subsection{Submodularity conditions and properties of the best-response-map}
Existence of MFG solutions is subject to the following abstract structural condition. 
\begin{assumption}[Submodularity Conditions]\ \label{assumption}
There exist operations $\land^{\text{\tiny{$E$}}}, \lor^{\text{\tiny{$E$}}}\colon  E \times E \to E$  such that:
\begin{enumerate} 
\item\label{assumption.existence.operation}  The projection $p$ behaves like a homeomorphism of lattices; that is, 
$$ 
p(\nu \land^{\text{\tiny{$E$}}} \bar{\nu} ) \leq^{\text{\tiny{$L$}}} p\nu \land^{\text{\tiny{$L$}}} p \bar{\nu} \leq^{\text{\tiny{$L$}}} p\nu \lor^{\text{\tiny{$L$}}} p\bar{\nu} \leq^{\text{\tiny{$L$}}} p(\nu \lor^{\text{\tiny{$E$}}} \bar{\nu}),\quad  \text{for each } \nu,\bar{\nu} \in E.
$$ 
\item \label{assumption.submod.J}
The cost functional satisfies the following submodularity properties  
$$
J(\nu \lor^{\text{\tiny{$E$}}} \bar{\nu} , \bar{\mu}) - J(\bar{\nu} , \bar{\mu}) \leq J(\nu \lor^{\text{\tiny{$E$}}} \bar{\nu} , {\mu}) - J( \bar{\nu}, {\mu}) \leq J(\nu , {\mu}) - J( \nu \land^{\text{\tiny{$E$}}} \bar{\nu} , {\mu}),
$$
for each $\nu, \bar{\nu} \in E$ and $\mu, \bar{\mu} \in L$ with $\mu \leq^{\text{\tiny{$L$}}}  \bar{\mu}$. 
\end{enumerate}
\end{assumption}

We underline that Condition \ref{assumption.submod.J} in Assumption \ref{assumption} coincides with the conditions in  \cite{To} only in the case in which $(E,\land^{\text{\tiny{$E$}}},\lor^{\text{\tiny{$E$}}})$ is a lattice.

We start our analysis with the following result on the structure of the sets of best responses.
\begin{lemma}\label{lemma.bestresponse}
Under Assumptions \ref{assumption.bestresponsemap.nonempty} and \ref{assumption} we have that:
\begin{enumerate}
\item[a)]\label{lemma.bestresponse.ordered} The set $R(\mu)$ is directed, i.e., for every $\eta^1,\eta^2\in R(\mu)$, there exist $\eta^\wedge,\eta^\vee\in R(\mu)$ such that $\eta^\wedge \leq^{\text{\tiny{$L$}}} \eta^1 \wedge^{\text{\tiny{$L$}}} \eta^2$ and $\eta^\vee \geq^{\text{\tiny{$L$}}} \eta^1 \vee^{\text{\tiny{$L$}}} \eta^2$.
 \item[b)]\label{lemma.bestresponse.increasing} For all $\mu,\overline{\mu}\in L$ with $\mu\leq^{\text{\tiny{$L$}}} \overline{\mu}$, $\inf R(\mu)\leq^{\text{\tiny{$L$}}} \inf R(\overline{\mu})$ and $\sup R(\mu)\leq^{\text{\tiny{$L$}}} \sup R(\overline{\mu})$.
 \item[c)]\label{lemma.bestresponsecomplete} For every $\mu\in L$, $\inf R(\mu)\in R(\mu)$ and $\sup R(\mu)\in R(\mu)$.
\end{enumerate}
\end{lemma}

\begin{proof}\ 
 Let $\mu,\overline{\mu}\in L$ with $\mu\leq^{\text{\tiny{$L$}}} \overline{\mu}$. Moreover, let $\eta^1\in R(\mu)$ and $\eta^2\in R(\overline \mu)$. Then, by definition of $R(\mu)$ and $R(\overline \mu)$, there exists $\nu^1\in \argmin_{\nu\in E}J(\nu, \mu)$ and $\nu^2\in \argmin_{\nu\in E}J(\overline \mu,\nu)$ with $p\nu^1=\eta^1$ and $p\nu^2=\eta^2$. 
 By Condition \ref{assumption.existence.operation} in Assumption \ref{assumption}, we can define $\nu^\wedge, \nu^\vee  \in E$ by $\nu^\wedge:= \nu^1 \wedge^{\text{\tiny{$E$}}} \nu^2$ and $\nu^\vee:= \nu^1 \vee^{\text{\tiny{$E$}}} \nu^2$, leading to $p\nu^\wedge \leq^{\text{\tiny{$L$}}} p\nu^1 \wedge^{\text{\tiny{$L$}}} p\nu^2$ and $p\nu^\vee \geq^{\text{\tiny{$L$}}} p\nu^1 \vee^{\text{\tiny{$L$}}} p\nu^2$. The optimality of $\nu^1$ and $\nu^2$ for $\eta^1$ and $\eta^2$, respectively, together with Condition \ref{assumption.submod.J} in Assumption \ref{assumption}, imply that
 \[
   0\leq J(\overline\mu,\nu^\vee)-J(\overline\mu,\nu^2)\leq J(\mu,\nu^\vee)-J(\mu,\nu^2)\leq J(\mu,\nu^1)-J(\mu,\nu^\wedge)\leq 0.
  \]
  This shows that $\eta^\wedge:=p\nu^\wedge\in R(\mu)$ and $\eta^\vee:=p\nu^\vee\in R(\overline\mu)$. Now, the statement in a) directly follows by choosing $\mu=\overline\mu$. Moreover, $\eta^\wedge\in R(\mu)$ and $\eta^\vee\in R(\overline\mu)$ imply that
  \[
   \inf R(\mu)\leq^{\text{\tiny{$L$}}} \eta^\wedge =p\nu^\wedge \leq^{\text{\tiny{$L$}}} p\nu^2=\eta^2\quad \text{and}\quad \eta^1=p\nu^1 \leq^{\text{\tiny{$L$}}}  p\nu^\vee =\eta^\vee \leq^{\text{\tiny{$L$}}} \sup R(\overline\mu). 
  \]
  Taking the infimum over all $\eta^2\in R(\overline\mu)$ and the supremum over all $\eta^1\in R(\mu)$ yields the assertion in b).
  
  We now prove the claim in c) for the infimum. Since $L$ is, by assumption, Dedekind super complete, there exists a sequence $(\mu^n)_n \subset R(\mu)$ such that $\inf R(\mu)= \inf_n \mu^n$. Therefore we can find a sequence $(\nu^n)_n \subset \argmin J(\cdot, \mu)$ with $p \nu^n= \mu^n$. We can inductively define a new sequence $(\nu^{\land,n})_n$ by setting 
  $$
  \nu^{\land,1}:= \nu^1 \quad \text{and}\quad
  \nu^{\land, n+1} := \nu^{\land, n} \land^{\text{\tiny{$E$}}} \nu^{n+1}, \quad n\geq 1. 
  $$
 As shown in the proof of part a), we have that $\nu^{\land, 1} \in \argmin J(\cdot, \mu)$, and by induction, we deduce that $\nu^{\land, n} \in \argmin J(\cdot, \mu)$ for each $n \in \mathbb{N}$. 
 Define now the sequence $(\mu^{\land,n})_n$ setting, $\mu^{\land, n}:= p\nu^{\land,n}$ for each $n \in \mathbb{N}$,
 and note that $\mu^{\land,n} \in R(\mu)$. 
 Moreover, Condition \ref{assumption.existence.operation} in Assumption \ref{assumption} implies that
 $$ 
\mu^{\land,n+1} = p(\nu^{\land,n+1}) = p(\nu^{\land, n} \land^{\text{\tiny{$E$}}} \nu^{n+1} ) \leq^{\text{\tiny{$L$}}} p\nu^{\land, n} \land^{\text{\tiny{$L$}}} p \nu^{ n+1} = \mu^{\land,n} \land^{\text{\tiny{$L$}}} \mu^{n+1},
$$
which, at the same time, implies that $(\mu^{\land,n})_n$ is nonincreasing in $L$ and that $\mu^{\land,n} \leq^{\text{\tiny{$L$}}} \mu^n$ for each $n\in \mathbb{N}$. Hence we have
$$
\inf R(\mu)= \inf_n \mu^n = \inf_n \mu^{\land,n}.
$$
Moreover, by Assumption \ref{assumption.bestresponsemap.nonempty}, there exists a subsequence $(n_j)_{j \in \N}$ and a limit point $\nu \in \argmin_E J(\cdot, \mu)$ such that 
$$
p\nu = \inf _j p(\nu^{\land, n_j}) = \inf_j \mu^{\land, n_j} = \inf R(\mu), 
$$
so that $\inf R(\mu) \in R(\mu)$.
\end{proof}

\subsection{Existence and approximation of MFG solutions}
For the approximation of MFG solutions, we will enforce the following additional continuity requirements (see again Remark \ref{remark assumptions non topological} for a generalization).
\begin{assumption}
\label{assumption.general.approximation}
For any sequence  $(\nu^n)_n \subset \{ \argmin _E J(\cdot, \mu)\,|\, \mu \in L \}$ such that $p (\nu^n)$ is increasing or decreasing in $L$, there exists a subsequence $(n_j)_{j \in \N}$ and $\nu \in E$ such that $\nu^{n_j}$ converges to $\nu$ as $j \to \infty$ and $ p\nu = \sup_j p(\nu^{n_j})$ or $p\nu = \inf_j p(\nu^{n_j})$, respectively. 
    
Moreover, respectively, for any nondecreasing or nonincreasing sequence $(\mu^n)_n \subset L$, we assume that
\begin{enumerate}
    \item for any $\nu\in E$, $J(\nu,\sup_n \mu) = \lim_n J(\nu, \mu^n)$ or $J(\nu,\inf_n \mu) = \lim_n J(\nu,\mu^n)$), 
    \item for any sequence $(\nu^n)_n \subset E$ converging to $\nu \in E$, we have that $J(\nu,\sup_n \mu) \leq \liminf_n J(\nu^n,\mu^n)$ or $J(\nu,\inf_n \mu) \leq \liminf_n J(\nu^n,\mu^n)$.  
\end{enumerate}
\end{assumption}

We can then state the main result of this section.

\begin{theorem}
\label{theorem.main.general} Under Assumptions \ref{assumption.bestresponsemap.nonempty} and \ref{assumption} we have that
 \begin{enumerate}
  \item[a)]\label{theorem.main.general.existence} the set of mean field game equilibria $M$ is nonempty with $\inf M\in M$ and $\sup M\in M$. If $R(\mu)$ is a singleton for all $\mu\in L$, then $M$ is a nonempty complete lattice.
  \end{enumerate}
  Moreover, if Assumption \ref{assumption.general.approximation} is satisfied, then
  \begin{enumerate}
  \item[b)]\label{theorem.main.general.convergence.up} the learning procedure $\underline{\mu}^0:=\inf L$ and $\underline{\mu}^n:=\inf  R(\underline{\mu}^{n-1})$, for $n\in \N$, is monotone increasing and it converges to $\inf M$, 
  \item[c)]\label{theorem.main.generalconvergence.down} the learning procedure $\overline{\mu}^0:=\sup L$ and $\overline{\mu}^n:=\sup R(\overline{\mu}^{n-1})$, for $n\in \N$, is monotone decreasing and it converges to $\sup M$.
 \end{enumerate}
\end{theorem}

\begin{proof}\
 \begin{enumerate}
  \item[a)] Follows directly from Lemma \ref{lemma.bestresponse} together with Tarski's fixed point theorem applied to the maps $\mu\mapsto \inf R(\mu)$ and $\mu\mapsto \sup R(\mu)$.
  \item[b)] By Lemma \ref{lemma.bestresponse}, it follows that the sequence $(\underline\mu^n)_{n\in \N_0}$ is increasing. 
  By completeness of the lattice $L$, we can set $\mu_* := \sup_n \underline \mu ^n$. We next want to prove that $\mu_* = \inf M$.
  
  For any $n\in \N$, by Lemma \ref{lemma.bestresponse.increasing} and the definition of $\underline\mu^n$, we can find $$\nu^n \in \argmin_E J( \cdot, \underline \mu ^{n-1})$$ with $p \nu^n = \underline \mu ^n$.
  By Assumption \ref{assumption.general.approximation}, we can take a subsequence $(\nu^{n_j})_j$ and a limit point $\nu_*$ such that  $ \nu^{n_j}$ converges to $\nu_*$ and $ p \nu^{n_j}$  converges to $p \nu_*$ as $j \to \infty$. This implies that $p \nu_* = \mu_*$. Moreover, we have
  \begin{equation*}
    J(\nu^{{n_j}}, \underline \mu ^{n_j -1}) \leq J(\nu, \underline \mu^{n_j -1}), \quad \text{for any $\nu \in E$ and $j \in \mathbb N$}.
  \end{equation*}
  Exploiting the continuity properties of $J$ in Assumption \ref{assumption.general.approximation}, we may pass to the limit as $j \to \infty$ in the previous inequality, and obtain that
  \begin{equation*}
    J(\nu_*, \mu_*) \leq J(\nu, \mu_*), \quad \text{for any $\nu \in E$ and $j \in \mathbb N$}.
  \end{equation*}
  This, in turn, implies that $\nu_* \in \argmin_E J(\cdot, \mu_*)$, so that  $\mu_* = p \nu_* \in R(\mu_*)$. Therefore, $\mu_*$ is a MFG solution.
  
  We next want to prove that $\mu_*$ is the minimal MFG solution.
  Let $\mu\in M$ be another mean field game equilibrium. 
  Then, $\underline\mu^0 \leq^{\text{\tiny{$L$}}} \mu$, which, by Lemma \ref{lemma.bestresponse}, implies that $\underline\mu^1=R(\underline\mu^0) \leq^{\text{\tiny{$L$}}} R(\mu)$. 
  Inductively, one obtains that $\underline\mu^n \leq^{\text{\tiny{$L$}}} \mu$ for all $n\in \N_0$, which implies that $\mu_* \leq^{\text{\tiny{$L$}}} \mu$. 
  Since $\mu_*\in M$, it follows that $\mu_*=\inf M$.
  \item[c)] Follows by arguments analogous to the one used in the proof of part b).
 \end{enumerate}
\end{proof}

The following remark proposes a set of purely order-theoretical conditions, alternative to those in Assumptions \ref{assumption.bestresponsemap.nonempty} and \ref{assumption.general.approximation}, respectively.
These will be employed in the proof of Proposition \ref{FiniteMarkovProp}.

\begin{remark}\label{remark assumptions non topological}
The proofs of Lemma \ref{lemma.bestresponse} and Theorem \ref{theorem.main.general} show that all stated properties remain valid if Assumption \ref{assumption.bestresponsemap.nonempty} is replaced by the following purely order-theoretic assumptions:
\begin{itemize}
    \item  For every $\mu\in L$, the set $\argmin_{E} J(\cdot, \mu)$ is nonempty and the set $R(\mu)$ is closed under monotone sequences; 
    that is, for any nondecreasing or nonincreasing sequence  $(\mu^n)_n \subset R(\mu)$, there exists $\nu \in \argmin_{E} J(\cdot, \mu)$ such that $ p\nu = \sup_n \mu^n$ or $p\nu = \inf_n \mu^n$, respectively;
\end{itemize}
and if Assumption \ref{assumption.general.approximation} is replaced by the following two order-theoretic conditions:
\begin{itemize}
    \item  For any $\nu\in E$, $J(\nu,\cdot)$ is continuous over monotone sequences in $L$;
    \item For any sequence $(\nu^n, \mu^n)_n \in E \times L$ such that $p \nu^n$ and $\mu^n$ are nondecreasing or nonincreasing, there exist $\nu \in E$ such that $p\nu = \sup_n p\nu^n$ or $p\nu = \inf_n p\nu^n$
    and $J(\nu,\sup_n \mu^n) \leq \liminf_n J(\nu^n,\mu^n)$ or $J(\nu,\inf_n \mu^n) \leq \liminf_n J(\nu^n,\mu^n)$, respectively.
\end{itemize}
\end{remark}


\section{Lattices of measures related to submodular MFG}
\label{section lattices}

In this section, we discuss lattices of measures arising in the context of submodular MFGs. Again, we refer to the Appendix \ref{append.lattice} for the lattice-theoretic preliminaries. Throughout this section, let $\mathcal B(\R)$ denote the Borel $\sigma$-algebra on $\R$ and $\mathcal M_{\leq 1}$ denote the set of all sub-probability measures, i.e., the set of all (nonnegative) measures on $\mathcal B(\R)$ with $\mu(\R)\leq 1$. We identify a distribution $\mu\in \mathcal M_{\leq 1} $ by its survival function $\mu_0$, i.e., we identify
\[
 \mu(s)=\mu_0(s):=\mu\big((s,\infty)\big)\quad \text{for all }s\in \R.
\]
 On $\mathcal M_{\leq 1} $, we consider the partial order $\leq_{\rm st}$ arising from \textit{first order stochastic dominance}, given by
\[
 \mu\leq_{\rm st}\nu \quad \text{if and only if}\quad \mu_0(s)\leq \nu_0(s)\quad \text{for all }s\in \R.
\]
Recall that, for $\mu,\nu\in \mathcal M_{\leq 1} $, $\mu\leq_{\rm st}\nu$ if and only if
\begin{equation}\label{eq.stochorder}
 \int_\R h(x)\, {\rm d}\mu (x)\leq \int_\R h(x)\, {\rm d}\nu(x)
\end{equation}
for all nondecreasing functions $h\colon \R\to [0,\infty)$. In particular, $\mu(\R)\leq \nu(\R)$. Note that \eqref{eq.stochorder} holds for all nondecreasing functions $h\colon \R\to \R$ if and only if $\mu\leq_{\rm st}\nu$ and $\mu(\R)=\nu(\R)$. For a detailed discussion on the properties of the partial order $\leq_{\rm st}$ for probability measures, we refer to \cite[Section 1.A]{ShakedShanthikumar07}.

By identifying a sub-probability measure $\mu$ with its survival function $\mu_0$, the set $\mathcal M_{\leq 1} $ coincides with the set of all nonincreasing right-continuous functions $F\colon \R\to [0,\infty)$ with $\lim_{s\to -\infty} F(s)\leq 1$ and $\lim_{s\to \infty} F(s)=0$. In particular, the partial order $\leq_{\rm st}$ induces a lattice structure on $\mathcal M_{\leq 1} $ via
 \[
  \qquad \quad \big(\mu\vee_{\rm st}\nu\big)(s):=\mu_0(s)\vee \nu_0(s)\quad \text{and}\quad \big(\mu\wedge_{\rm st}\nu\big)(s):=\mu_0(s)\wedge \nu_0(s)\quad \text{for all }s\in \R.
 \]
 Further, we would like to recall that the weak convergence coincides with the pointwise convergence of survival functions at every continuity point, i.e., $\mu^n\to \mu$ weakly as $n\to \infty$ if and only if
 \[
  \qquad\mu_0^n(s)\to \mu_0(s)\quad\text{as }n\to \infty\quad \text{for every continuity point }s\in \R\text{ of }\mu_0,
 \]
 and that the weak topology  on $\mathcal M_{\leq 1} $, i.e., the topology induced by the weak convergence of sub-probability measures, is metrizable.
 As a consequence, the lattice operations $(\mu,\nu)\mapsto \mu \vee_{\rm st} \nu $ and $(\mu,\nu)\mapsto \mu \wedge_{\rm st} \nu $ are continuous maps $\mathcal M_{\leq 1} \times \mathcal M_{\leq 1} \to \mathcal M_{\leq 1} $, and the weak topology is finer than the interval topology (see Definition \ref{def.intervaltop} in Appendix \ref{append.lattice}), since every closed interval is weakly closed.
 
\begin{lemma}\label{minimaandmaxima}
Every bounded and nondecreasing or nonincreasing sequence $(\mu^n)_{n\in \N}\subset \mathcal M_{\leq 1} $ converges weakly to its supremum or infimum w.r.t.~$\leq_{\rm st}$, respectively.
\end{lemma}

\begin{proof}
First, observe that a nonincreasing function $\R\to \R$ is right-continuous if and only if it is lower semicontinuous. Hence, for every sequence $(\mu^n)_{n\in \N}\in \mathcal M_{\leq 1} $, which is bounded above, the supremum $\sup_{n\in \N}\mu^n$ w.r.t.~$\leq_{\rm st}$ exists, and it is exactly the pointwise supremum of the survival functions $(\mu_0^n)_{n \in \mathbb{N}}$.

For a nonincreasing function $F\colon \R\to \R$, we define its lsc-envelope $F_*\colon \R\to \R$ by
 \[
  F_*(s):=\sup_{\delta >0}F(s+\delta)\quad \text{for }s\in \R.
 \]
 Then, $F(s)\geq F_*(s)\geq F(s+\varepsilon)$ for all $s\in \R$ and $\varepsilon >0$. That is, $F_*$ differs from $F$ only at discontinuity points of $F$. For a sequence $(\mu^n)_{n\in \N}\in \mathcal M_{\leq 1}$, which is bounded below, the infimum $\inf_{n\in \N}\mu^n$ w.r.t.~$\leq_{\rm st}$ is then given by the lsc-envelope of the pointwise infimum of the survival functions $(\mu_0^n)_{n\in \N}$. Since the weak convergence of a sequence of sub-probability measures coincides with the pointwise convergence of the related survival functions at every continuity point, the assertion follows.
\end{proof}

Let $(S,\mathcal S,\pi)$ be a $\sigma$-finite measure space. We denote the Borel $\sigma$-algebra of the weak topology by $\mathcal B(\mathcal M_{\leq 1} )$ and the lattice of all equivalence classes of $\mathcal S$-$\mathcal B(\mathcal M_{\leq 1} )$-measurable functions $S\to \mathcal M_{\leq 1} $ by $L_{\rm st}^0=L^0(S,\mathcal S,\pi;\mathcal M_{\leq 1})$. An arbitrary element $\mu$ of $L_{\rm st}^0$ will be denoted in the form $\mu=(\mu_t)_{t\in S}$. On $L_{\rm st}^0$ we consider the order relation $\leq_{L_{\rm st}^0}$, given by $\mu \leq_{L_{\rm st}^0} \nu$ if and only if $\mu_t \leq_{\rm st} \nu_t$ for $\pi$-a.a.\ $t\in S$.

In the sequel, we consider a family $(L_n)_{n\in \N}$ of Dedekind $\sigma$-complete sublattices of $\mathcal M_{\leq 1}$, which correspond to a countable number of constraints, and a family $(B_n)_{n\in \N}\subset \mathcal S$ of measurable sets, on which the constraints in terms of the family $(L_n)_{n\in\N}$ should be satisfied. Before we state the main result of this section, we list some possible choices for measurable spaces $(S,\mathcal S,\pi)$, Dedekind $\sigma$-complete lattices $L=L_n$, and measurable sets $B=B_n\in \mathcal S$ for $n\in \N$.

\begin{example}\label{ex.lattices}\
 \begin{enumerate}
 \item[a)] The measure space $(S,\mathcal S,\pi)$ can be, e.g.,
 \begin{itemize}
  \item  $S=[0,T]$, $\mathcal S=\mathcal B([0,T])$, $\pi=\delta_0+\lambda_{[0,T]}$, also with $[0,\infty)$ instead of $[0,T]$ and $e^{-\delta t}{\rm d}t$ instead of $\lambda$,
  \item $\Omega\times [0,T]$, $\mathcal S$ the $\sigma$-algebra of all predictable processes, and $\pi=\P\otimes (\delta_0+\lambda_{[0,T]})$.
 \end{itemize}
 \item[b)] The following are possible choices for $L=L_n$:
 \begin{itemize}
 \item The simplest choice is $L=\mathcal M_{\leq 1}$ or $L=\{\mu\in \mathcal M_{\leq 1}\, |\, \mu(\R)=1\}$.
  \item Another choice is $L=\{\mu\in \mathcal M_{\leq 1}\, |\, \underline \mu\leq_{\rm st}\mu \leq_{\rm st}\overline \mu\}$ with $\underline \mu,\overline \mu\in\mathcal M_{\leq 1}$.\ If $\underline\mu=\overline\mu=:\nu$, this results in $L=\{\nu\}$. Note that $\underline \mu\equiv 0$ is not excluded.
  \item Let $a,b\in \R$ with $a\leq b$. Then, $L=\{\mu\in \mathcal M_{\leq 1}\, |\, \supp \mu\subset [a,b]\}$ is Dedekind $\sigma$-complete. In fact, a sub-probability $\mu\in \mathcal M_{\leq 1}$ is an element of $L$ if and only if its survival function $\mu_0$ is constant on $(-\infty,a)$ and $(b,\infty)$, a property that carries over to suprema and infima of countably many elements of $L$. The same holds true if the interval $[a,b]$ is replaced by $[a,\infty)$ or $(-\infty,b]$.
  \item Another possible choice is $L=\{\delta_x\, |\, x\in \R\}$ (Dirac measures).
   \end{itemize}
 \item[c)] Possible choices for $B=B_n$ are
 \begin{itemize}
  \item $B=\{0\}$ or $B=\{0\}\times \Omega$ in order to prescribe an initial condition,
  \item $B=[0,T]$ or $B=\Omega\times [0,T]$ in order to give a condition that should be satisfied for all times $t\in [0,T]$ and in all states $\omega\in \Omega$,
  \item $B=A\times (t_1,t_2]$ in order to prescribe a condition on a certain event $A$ during the time period $ (t_1,t_2]$. 
 \end{itemize}
 \end{enumerate}
\end{example}


We consider the set
\[
 \mathcal L:=\big\{\mu\in L^0_{\rm st}\, \big|\, \forall n\in \N\colon \pi\big(\{t\in S\, |\, \mu_t\notin L_n\}\cap B_n\big)=0\big\}.
\]
That is, the set of all measurable flows $(\mu_t)_{t\in S}$ of sub-probability measures such that, for all $n\in \N$, $\mu_t\in L_n$ for $\pi$-a.a.\ $t\in B_n$. The following theorem is the main result of this section.

\begin{theorem}\label{thm.completeness}\
 \begin{enumerate}
  \item[a)] The lattice $\mathcal L$ is Dedekind super complete.
  \item[b)] If $M\subset \mathcal L$ is a nonempty set, which is bounded above or below and directed upwards or downwards, then there exist sequences $(\overline \mu^n)_{n\in \N}\subset M$ and $(\underline \mu^n)_{n\in \N}\subset M$ with $\overline \mu^n\leq_{L_{\rm st}^0} \overline\mu^{n+1}$ and $\underline\mu^n\geq_{L_{\rm st}^0} \underline\mu^{n+1}$ for all $n\in \N$ and $$\overline\mu^n\to \sup M\in \mathcal L \quad \text{and}\quad \underline\mu^n\to \inf M\in \mathcal L\quad \text{weakly }\pi\text{-a.e.\ as }n\to \infty,$$
  respectively.
 \end{enumerate}
\end{theorem}


\begin{proof}
 Since every $\sigma$-finite measure can be transformed to a probability measure without changing the null-sets, we may, w.l.o.g., assume that $\pi(S)=1$. By Remark \ref{minimaandmaxima}, and since the lattices $(L_n)_{n\in \N}$ are Dedekind $\sigma$-complete, $\mathcal L$ is Dedekind $\sigma$-complete. Let $\Phi\colon \R\to (0,1)$ be the cumulative distribution function of the standard normal distribution, i.e.
 \[
  \Phi(x):=\frac{1}{\sqrt{2\pi}}\int_{-\infty}^x e^{-y^2/2}\, {\rm d}y\quad\text{for all }x\in \R.
 \]
 The map $S\to \R, \; t\mapsto \int_\R\Phi(x)\, {\rm d}\mu_t(x)$ is $\mathcal S$-$\mathcal B(\R)$-measurable for every $\mu\in L_{\rm st}^0$, since the bounded and continuous function $\Phi\colon \R\to (0,1)$ induces a continuous (w.r.t.~the weak topology) functional $\mathcal M_{\leq 1} \to \R$. Hence,
 \[
  F\colon \mathcal L\to \R,\quad \mu\mapsto \int_S\int_\R\Phi(x)\, {\rm d}\mu_t(x)\,{\rm d}\pi(t)
 \]
 is well-defined and strictly increasing, since $\Phi$ is nonnegative and strictly increasing, see, e.g., \cite[Theorem 1.A.8]{ShakedShanthikumar07}. The assertions now follow from Lemma \ref{lem.blms} and Lemma \ref{minimaandmaxima}.
\end{proof}


\section{Submodular mean field games with Markov chains} \label{SectFiniteMarkov}

Throughout this section, let $d\in \mathbb{N}\setminus \{1\}$ and $S := \{1,\ldots,d\}$ be a finite state space. We endow $S$ with the natural order, and identify elements of the set $\mathcal{P} (S)$ of all probability measures with their probability vectors according to
\[
	\mu \equiv (\mu_{1},\ldots,\mu_{d}):= \bigl(\mu(\{1\}),\ldots,\mu(\{d\})\bigr),\quad \mu\in \mathcal{P}(S). 
\]
We consider probability vectors as row vectors. On $\mathcal{P}(S)$, we introduce a partial order $\preceq$ through
\begin{align*}
	& \mu \preceq \nu \quad \text{if and only if} \quad  \sum_{i=1}^{l} \mu_{i} \geq \sum_{i=1}^{l} \nu_{i} \quad \text{for all } l\in \{1,\ldots,d\}.
\end{align*}
This corresponds to the usual stochastic order in terms of cumulative distribution functions when interpreting $S = \{1,\ldots,d\}$ as a subset of $\mathbb{R}$ with the natural order. 
As a consequence, we have
\begin{equation} \label{FiniteMarkovMeans}
	\sum_{i=1}^{d} c_{i} \mu_{i} \geq \sum_{i=1}^{d} c_{i} \nu_{i} \text{ whenever } \mu \preceq \nu \text{ and } S\ni i \mapsto c_{i} \in \mathbb{R} \text{ is nonincreasing},
\end{equation}
see, for instance, \cite[Section~1.A.1]{ShakedShanthikumar07}.

For $\mu, \nu \in \mathcal{P}(S)$, their greatest lower bound $\mu \wedge \nu$ and least upper bound $\mu \vee \nu$, respectively, are given by
\begin{align*}
	& \left( \mu \wedge \nu \right)_{j}:=\max\left\{ \sum_{k=1}^{j} \mu_{k}, \sum_{k=1}^{j} \nu_{k} \right\}-\max\left\{ \sum_{k=1}^{j-1} \mu_{k}, \sum_{k=1}^{j-1} \nu_{k} \right\}\quad \text{and}\\
	& \left( \mu \vee \nu \right)_{j} :=\min\left\{ \sum_{k=1}^{j} \mu_{k}, \sum_{k=1}^{j} \nu_{k} \right\}-\min\left\{ \sum_{k=1}^{j-1} \mu_{k}, \sum_{k=1}^{j-1} \nu_{k} \right\} \quad \text{for all }j\in \{1,\ldots,d\},
\end{align*}
where we use the convention $\sum_{k=1}^0\mu_k:=0$ and $\sum_{k=1}^0\nu_k:=0$. Then, $(\mathcal{P}(S), \preceq)$ is a complete lattice.

We consider a fixed finite time horizon $T\in \N$ and a fixed initial distribution $\eta \in \mathcal{P}(S)$. Let $L$ be the set of all flows $$\mu\colon \{0,\ldots,T\} \to \mathcal{P}(S)\quad \text{with }\mu_0=\eta,$$
and let $\leq^{L}$ be the partial order on $L$ induced by $\preceq$, that is,
\begin{align*}
 \mu \leq^{L} \nu \quad \text{if and only if}\quad\mu_t \preceq \nu_t \quad \text{for all } t\in \{0,\ldots,T\}.
\end{align*}
The greatest lower bound $\mu\wedge^L\nu$ and the least upper bound $\mu\vee^L \nu$ of two elements $\mu, \nu\in L$ are then given by
\begin{align*}
 \left( \mu \wedge^{L} \nu \right)_t := \mu_t \wedge \nu_t\quad \text{and}\quad \left( \mu \vee^{L} \nu \right)_t := \mu_t \vee \nu_t \quad\text{for all } t\in \{0,\ldots,T\}. 
\end{align*}
Observe that $(L, \leq^{L})$ is again a complete lattice.

Let $\Gamma$ be a non-empty set; $\Gamma$ represents the set of control actions for the representative player. Define the set $\mathcal U$ of $\Gamma$-valued open-loop strategies as the set of all mappings $u\colon \{0,\ldots,T-1\} \to \Gamma$.

Let $A(\gamma))_{\gamma \in \Gamma}$, be a family of transition matrices on $S$. Thus, for each $\gamma \in \Gamma$, $A(\gamma) = (a_{ij}(\gamma))_{i,j\in S}$ is a $d\times d$-matrix with nonnegative entries such that
\[
	\sum_{j=1}^{d} a_{ij}(\gamma) = 1\quad \text{for all } i\in S.
\]
For $u\in \mathcal U$, we define the flow $\mu^u$ of laws of the controlled Markov chain recursively through
\begin{align} \label{FiniteMarkovDynamics}
 \mu^u_0:=\eta\quad\text{and}\quad \mu_{t+1}^u:= \mu_t^u A(u_t)\quad \text{for all }t\in \{0,\ldots, T-1\},
\end{align}
where $\eta \in \mathcal{P}(S)$ is the fixed initial distribution and we recall that elements of $\mathcal{P}(S)$ are identified as row vectors.

Let $E$ be the subset of $\mathcal{U} \times L$ given by
\[
	E := \left\{ (u,\mu^u) \colon  u \in \mathcal{U} \right\},
\]
 and let $p\colon E \rightarrow L$ be the projection on the second component:
\[
	p(u,\mu) :=\mu\quad \text{for all }(u,\mu) \in E.
\]
Thus, $p(u,\mu^u) = \mu^u$ for all $u\in \mathcal U$.

Let $f\colon \{0,\ldots,T-1\}\times S\times \mathcal{P}(S) \times \Gamma \to \mathbb{R}$, $g\colon S \times \mathcal{P}(S) \to \mathbb{R}$ be functions, representing the running and terminal costs, respectively. Define a functional $J\colon E \times L \to \mathbb{R}$ according to
\[
	J\big((u, \mu^u), \mu\big) := \sum_{t=0}^{T-1} \sum_{i=1}^{d} f\bigl(t,i,\mu_t,u_t \bigr) \mu^u_{t,i} + \sum_{i=1}^{d} g\bigl(i,\mu_T\bigr) \mu^u_{T,i},
\]
where, for $\mu\in L$, $t\in \{0,\ldots,T\}$, and $i\in S$, $\mu_{t,i}$ denotes the $i$-th coordinate of $\mu_t$.

As in Section~\ref{section general case}, we define the best response map $R\colon L \rightarrow 2^{L}$ according to
\[
	R(\mu):=\left\{ p(\nu) \colon\nu\in \argmin\nolimits_{E} J(\cdot,\mu) \right\}.
\]

The following conditions on the solution map and $J$ will entail the assumptions of the general setup:
 
\begin{assumption}[Sufficient conditions] \label{FiniteMarkovSufficient}
Suppose that $\leq^{\mathcal{U}}$ is a partial order on $\mathcal{U}$ making $(\mathcal{U}, \leq^{\mathcal{U}})$ a complete lattice such that: 
\begin{enumerate}
	\item \label{FiniteMarkovSufficient.compatibility} For every sequence $(u_{n})_{n\in \mathbb{N}} \subseteq \mathcal{U}$,
	\begin{align*}
		& \inf_{n\in \N} \mu^{u_{n}} = \mu^{u^\wedge}\quad \text{with }u^\wedge=\inf_{n\in \N}u_n\quad \text{and}\\ 
		 &\sup_{n\in \N}\mu^{u_{n}} = \mu^{u^\vee}\quad \text{with }u^\vee=\sup_{n\in \N}u_n.
	\end{align*}

	\item \label{FiniteMarkovSufficient.submod}
	For all $\hat{\mu}, \check{\mu}\in L$ and $\hat{u}, \check{u} \in \mathcal{U}$ with  $\hat{\mu} \leq^{L} \check{\mu}$, $\hat{u} \leq^{\mathcal{U}} \check{u}$,
	\[
	J\left( (\check{u}, \mu^{\check{u}}), \check\mu \right) - J\left( (\hat{u}, \mu^{\hat{u}}), \check\mu \right) \leq  J\left( (\check{u}, \mu^{\check u}), \hat \mu \right) - J\left( (\hat{u}, \mu^{\hat{u}}), \hat{\mu} \right) .
	\]

	\item \label{FiniteMarkovSufficient.costinfsup}
	Given any $\mu \in L$, we have for all sequences $(u_{n})_{n\in \mathbb{N}} \subseteq \mathcal{U}$, with $u^\wedge:=\inf_{n\in \N}u_n$ and $u^\vee:=\sup_{n\in \N}u_n$,
	\begin{align*}
		J\left( \big(u^\wedge,\mu^{u^\wedge}\big), \mu \right) &= \inf_{n\in \mathbb{N}} J\left( (u_{n}, \mu^{u_n}), \mu \right)\quad\text{and} \\
		J\left( \big(u^\vee,\mu^{u^\vee}\big), \mu \right) &= \sup_{n\in \mathbb{N}} J\left(( u_{n}, \mu^{u_n}), \mu \right),
	\intertext{or else}
		J\left( \big(u^\wedge,\mu^{u^\wedge}\big), \mu \right) &= \sup_{n\in \mathbb{N}} J\left( (u_{n}, \mu^{u_n}), \mu \right)\quad\text{and} \\
		J\left( \big(u^\vee,\mu^{u^\vee}\big), \mu \right) &= \inf_{n\in \mathbb{N}} J\left(( u_{n}, \mu^{u_n}), \mu \right).
	\end{align*}
\end{enumerate}
\end{assumption}

\begin{proposition} \label{FiniteMarkovProp}
	Given Assumption~\ref{FiniteMarkovSufficient}, the set $E$ together with the pointwise lattice operations 
	\begin{align*}
	(u,\mu^u) \wedge^{E} (v,\mu^v):= (u \wedge^{\mathcal{U}} v, \mu^u \wedge^{L} \mu^v)\quad\text{and}\quad(u,\mu^u) \vee^{E} (v,\mu^v):= (u \vee^{\mathcal{U}}v, \mu^u \vee^{L} \mu^v),
	\end{align*}
 for $u,v\in \mathcal{U}$, becomes a lattice, and the best response map $R$, the projection $p$, and the cost functional $J$ satisfy Assumption \ref{assumption} and the alternative for Assumption~\ref{assumption.bestresponsemap.nonempty} from Remark \ref{remark assumptions non topological}.
\end{proposition}

\begin{proof}
First observe that, thanks to Condition \ref{FiniteMarkovSufficient.compatibility} in Assumption \ref{FiniteMarkovSufficient}, the operations $\wedge^{E}$, $\vee^{E}$ are well-defined in the sense that if $\nu, \bar{\nu} \in E$, then $\nu \wedge^{E} \bar{\nu}$ and $\nu \vee^{E} \bar{\nu}$ are again elements of $E$.

Since $\mu \wedge^{L} \bar{\mu} \leq^{L} \mu \vee^{L} \bar{\mu}$ for all $\mu, \bar{\mu} \in L$, we find that the projection $p$ satisfies Condition \ref{assumption.existence.operation} in Assumption \ref{assumption}. Indeed, if $(u,\mu^u), (v,\mu^v) \in E$, then
\begin{align*}
	p\bigl((u,\mu^u) \wedge^{E} (v,\mu^v)\bigr) &= \mu^u \wedge^{L} \mu^v = p\bigl((u,\mu^u)\bigr) \wedge^{L} p\bigl((v,\mu^v)\bigr) \leq p\bigl((u,\mu^u)\bigr) \vee^{L} p\bigl((v,\mu^v)\bigr) \\
	&= \mu^u \vee^{L} \mu^v = p\bigl((u,\mu^u) \vee^{E} (v,\mu^v)\bigr).
\end{align*}
Again, thanks to Condition \ref{FiniteMarkovSufficient.compatibility} in Assumption \ref{FiniteMarkovSufficient}, we have
\begin{equation} \label{FiniteMarkovProofSolMap}
	\mu^u \leq^{L} \mu^v\quad \text{for all } u, v \in \mathcal{U} \text{ with } u \leq^{\mathcal{U}}v,
\end{equation}
for in that situation, setting $u_{1}:= u$ and $u_{n}:=v$, for $n \in \mathbb{N}\setminus \{1\}$, we find $\mu^u = \mu^u \wedge^{L} \mu^v \leq^{L} \mu^v$.

Let $(u,\mu^u), (v, \mu^v) \in E$, and let $\hat\mu, \check\mu \in L$ be such that $\hat\mu \leq^{L} \check\mu$. Set $\check{u} :=u \vee^{\mathcal{U}} v$.
Then, thanks to Conditions \ref{FiniteMarkovSufficient.compatibility} and \ref{FiniteMarkovSufficient.submod} in Assumption \ref{FiniteMarkovSufficient},
\begin{align*}
	J\bigl((u,\mu^u) \vee^{E} (v,\mu^v), \check\mu \bigr) - J\bigl((v,\mu^v), \check\mu \bigr)&= J\left(( \check{u}, \mu^{\check{u}}), \check\mu \right) - J\big( (v, \mu^v), \check\mu \big) \\
	&\leq J\left( (\check{u}, \mu^{\check{u}}), \hat\mu \right) - J\big( (v,\mu^v), \hat\mu \big) \\ 
	&= J\bigl((u,\mu^u) \vee^{E} (v,\mu^v), \hat\mu \bigr) - J\bigl((v,\mu^v), \hat\mu \bigr).
\end{align*}
This establishes the first inequality in Condition \ref{assumption.submod.J} in Assumption \ref{assumption}. 
The second inequality in Condition \ref{assumption.submod.J} in Assumption \ref{assumption} is a consequence of Condition \ref{FiniteMarkovSufficient.costinfsup} in Assumption \ref{FiniteMarkovSufficient}. 
In fact, thanks to Condition \ref{FiniteMarkovSufficient.costinfsup} in Assumption \ref{FiniteMarkovSufficient}, we have for every $\mu \in L$ and all $(u,\mu^u), (v,\mu^v) \in E$,
\begin{align*}
	J\bigl((u,\mu^u) \vee^{E} (v,\mu^v)&, \mu \bigr) + J\bigl((u,\mu^u) \wedge^{E} (v,\mu^v), \mu \bigr) \\
	&= J\left( (u \vee^{\mathcal{U}}v, \mu^u \vee^{L}\mu^v), \mu \right) + J\left(( u \wedge^{\mathcal{U}} v, \mu^u \wedge^{L} \mu^v, \mu \right) \\
	&= \min\left\{J\left( (u, \mu^u), \mu \right), J\left( (v, \mu^v), \mu \right) \right\} + \max\left\{ J\left( (u, \mu^u), \mu \right), J\left( (v, \mu^v), \mu \right) \right\} \\
	&= J\left( (u, \mu^u), \mu \right) + J\left( (v,\mu^v), \mu \right).
\end{align*}
Let $\mu \in L$ and $((u_{n},\mu^{u_n}))_{n\in \mathbb{N}} \subset E$ be such that $J((u_{n},\mu^{u_n}),\mu) \searrow \inf_{\nu\in E} J(\nu,\mu)$ as $n \to \infty$. Notice that this infimum exists in $[-\infty,\infty)$. Set
\begin{align*}
	 \hat{u} :=\inf_{n\in \N} u_{n} \quad \text{and}\quad \check{u} :=\sup_{n\in\N} u_{n} .
\end{align*}
By Condition \ref{FiniteMarkovSufficient.compatibility} in Assumption \ref{FiniteMarkovSufficient}, we have
\begin{align*}
	\mu^{\hat{u}} = \inf_{n\in\N}\mu^{u_n}\quad\text{and}\quad \check{\mu} = \sup_{n\in\N}\mu^{u_n}.
\end{align*}
By Condition \ref{FiniteMarkovSufficient.costinfsup} in Assumption \ref{FiniteMarkovSufficient} (we only treat the first case there, the second is obtained by interchanging infima and suprema), we find that
\[
	J\bigl((\hat{u},\mu^{\hat{u}}), \mu \bigr)= \inf_{n\in \mathbb{N}} J\big((u_{n},\mu^{u_n}),\mu\big)\quad\text{and}\quad J\bigl((\check{u},\mu^{\check{u}}), \mu \bigr) =  \sup_{n\in \mathbb{N}} J\big((u_{n},\mu^{u_n}),\mu\big).
\]
It follows that $\inf_{\nu\in E} J(\nu,\mu) = J((\hat{u},\mu^{\hat{u}}), \mu)$, which shows that $(\hat{u},\mu^{\hat u}) \in \argmin_{E} J(\cdot ,\mu)$ and thus $\mu^{\hat{u}} \in R(\mu)$. In particular, the set of best response distributions is non-empty.


Now, suppose that $(\mu_{n})_{n\in \mathbb{N}} \subseteq R(\mu)$. For $n\in \mathbb{N}$, choose $u_{n} \in \argmin_{u\in \mathcal{U}} J((u,\mu^u),\mu)$ such that $\mu^{u_{n}} = {\mu}_{n}$. Define $\hat{u}$ and $\check{u}$ in the same way as above. By Condition \ref{FiniteMarkovSufficient.compatibility} in Assumption \ref{FiniteMarkovSufficient}, we have
\begin{align*}
 \mu^{\hat{u}} = \inf_{n\in\N} \mu_n\quad\text{and}\quad \mu^{\check{u}} = \sup_{n\in\N}\mu_n,
\end{align*}
and, by Condition \ref{FiniteMarkovSufficient.costinfsup} in Assumption \ref{FiniteMarkovSufficient} (in the first case there), we find again that
\[
J\bigl((\hat{u},\mu^{\hat{u}}), \mu \bigr)= \inf_{n\in \mathbb{N}} J\big((u_{n},\mu^{u_n}),\mu\big)\quad\text{and}\quad J\bigl((\check{u},\mu^{\check{u}}), \mu \bigr)=  \sup_{n\in \mathbb{N}} J\big((u_{n},\mu^{u_n}),\mu\big).
\]
But $u_{n} \in \argmin_{u\in \mathcal{U}} J((u,\mu^{u}),\mu)$ for every $n\in \mathbb{N}$, hence
\[
	\inf_{n\in \mathbb{N}} J\bigl((u_{n},\mu^{u_n}),\mu \bigr) = \sup_{n\in \mathbb{N}} J\bigl((u_{n},\mu^{u_n}),\mu \bigr).
\]
It follows that $\mu^{\hat{u}} \in R(\mu)$ as well as $\mu^{\check{u}} \in R(\mu)$. In particular, any monotone sequence in $R(\mu)$ has a limit in $R(\mu)$. We thus see that the alternative for  Assumption~\ref{assumption.bestresponsemap.nonempty} from Remark~\ref{remark assumptions non topological} is satisfied.
\end{proof}

By Proposition~\ref{FiniteMarkovProp}, Remark~\ref{remark assumptions non topological} and part a) of Theorem \ref{theorem.main.general}, one immediately obtains:

\begin{corollary} \label{FiniteMarkovMain}
	Given Assumption~\ref{FiniteMarkovSufficient}, the set $M$ of solutions to the finite state mean field game is nonempty and contains $\inf M$ as well as $\sup M$.
\end{corollary}

The following example shows a family of simple two-state models where the assumptions of Proposition~\ref{FiniteMarkovProp} are satisfied.

\begin{example}

	Choose $d = 2$, and set $\Gamma := [0,1]$ (with the natural order $\leq$). Choose $p, q \in (0,1]$ with $p \leq q$, and define controlled transition matrices $A(\gamma)$ according to
	\[
		A(\gamma) \doteq \begin{pmatrix} 1 - p\gamma & p\gamma \\
		1 - q\gamma & q\gamma \end{pmatrix}\quad\text{for all } \gamma \in \Gamma.
	\]
	With this choice, for all $\gamma \in \Gamma = [0,1]$ and all $\mu = (\mu_1,\mu_2) \in \mathcal{P}(S) = \mathcal{P}(\{1,2\})$,
	\[
		\mu A(\gamma) = \left( 1 - \gamma \bigl(p + \mu_2(q-p)\bigr), \gamma \bigl(p + \mu_2(q-p)\bigr) \right).
	\]
	For $\mu, \bar{\mu} \in \mathcal{P}(S)$, we have that $\mu \preceq \bar{\mu}$ if and only if $\mu_2 \leq \bar{\mu}_2$, and that
	\[
		\left\{ \mu \wedge \bar{\mu}, \mu \vee \bar{\mu} \right\} = \left\{ \mu, \bar{\mu} \right\}.	
	\]
	Therefore, if $\mu, \bar{\mu} \in L$, then for all $t \in \{0,\ldots,T\}$,
	\[
	\left\{ \left( \mu \wedge^{L} \bar{\mu} \right)_{t},  \left( \mu \vee^{L} \bar{\mu} \right)_{t} \right\} = \left\{ \mu_{t}, \bar{\mu}_{t} \right\}.	
	\]
	It also follows that
	\[
		\mu A(\gamma) \preceq \bar{\mu} A(\tilde{\gamma}) \quad\text{whenever } \gamma \leq \tilde{\gamma} \text{ and } \mu \preceq \bar{\mu}.
	\]
	In case $\mu = \bar{\mu}$, we have, for all $\gamma, \bar{\gamma} \in \Gamma$ and $(\gamma_{n})_{\in \mathbb{N}} \subset \Gamma$,
	\begin{align*}
		& \mu A(\gamma \wedge \bar{\gamma}) = \left(m A(\gamma) \right) \wedge \left(\mu A(\bar{\gamma}) \right), & & \mu A(\gamma \vee \bar{\gamma}) = \left(\mu A(\gamma) \right) \vee \left(\mu A(\bar{\gamma}) \right), & \\
		& \mu A\bigl( \inf_{n\in \mathbb{N}} \gamma_{n} \bigr) = \inf \left\{ \mu A(\gamma) : n\in \mathbb{N} \right\}, & & \mu A\bigl( \sup_{n\in \mathbb{N}} \gamma_{n} \bigr) = \sup \left\{ \mu A(\gamma) : n\in \mathbb{N} \right\}. &
	\end{align*}
	We introduce a partial order $\leq^{\mathcal{U}}$ on $\mathcal{U}$ by
	\begin{align*}
		& u \leq^{\mathcal{U}} \tilde{u} & &\text{if and only if}& & \mu^{u} \leq^{L} \mu^{\tilde{u}}. 
	\end{align*}
	Then, the greatest lower bound $u \wedge^{\mathcal{U}} v$ of two elements $u, v \in \mathcal{U}$ is defined as follows: Set $\hat{\mu} := \mu^{u} \wedge^{L} \mu^{v}$, and define $\left( u \wedge^{\mathcal{U}} v \right)_0:= \min\left\{ u_0, v_0 \right\}$ and, for $t\in \{0,\ldots,T-2\}$, 
	\[
		\left( u \wedge^{\mathcal{U}} v \right)_{t+1} :=  \min \left\{ u_t\cdot \frac{p+(\mu^{u}_{t})_{2}(q-p)}{p+(\hat{\mu}_{t})_{2}(q-p)},\; v_t\cdot \frac{p+(\mu^{v}_{t})_{2}(q-p)}{p+(\hat{\mu}_{t})_{2}(q-p)} \right\}.
	\]
	By induction, one checks that, for every $t\in \{0,\ldots,T-1\}$,
	\begin{align*}
		& \left( u \wedge^{\mathcal{U}} v \right)_t \in [0,1], & \mu^{u \wedge^{\mathcal{U}} v}_{t} = \hat{\mu}_{t}. &
	\end{align*}
	Indeed, the claim holds for $t =0$. Now, suppose that it holds up to time $t$ and that $\mu^{u}_{t+1} = \hat{\mu}_{t+1}$. Then $\hat{\mu}_{t+1} = \mu^{u}_t A(u(t))$ and there exists $\tilde{\gamma} \in \{u(t), \tilde{u}(t)\}$ such that  $\hat{\mu}_{t+1} \preceq \hat{\mu}_{t} A(\tilde{\gamma})$. But then
	\[
		(\hat{\mu}_{t+1})_{2} = u(t) \bigl(p + (\mu^{u}_{t})_{2}(q-p)\bigr) \leq \tilde{\gamma} \bigl(p + (\hat{\mu}_{t})_{2}(q-p)\bigr),
	\]
	hence
	\[
		 0 \leq u(t)\cdot \frac{p+(\mu^{u}_{t})_{2}(q-p)}{p+(\hat{\mu}_{t})_{2}(q-p)} \leq \tilde{\gamma} \leq 1.
	\]
	Moreover,
	\[
		\left( \hat{\mu}_{t} A\left( u(t)\cdot \frac{p+(\mu^{u}_{t})_{2}(q-p)}{p+(\hat{\mu}_{t})_{2}(q-p)} \right) \right)_{2} = u(t)\cdot \frac{p+(\mu^{u}_{t})_{2}(q-p)}{p+(\hat{\mu}_{t})_{2}(q-p)}\cdot \bigl(p + (\hat{\mu}_t)_{2}(q-p) \bigr) = (\hat{\mu}_{t+1})_{2}
	\]
	since $\hat{\mu}_{t+1} = \mu^{u}_t A(u(t))$ by assumption. The case $\hat{\mu}_{t+1} = \mu^{\tilde{u}}_t A(\tilde{u}(t))$ is handled in the same way.
	
	In analogy with the greatest lower bound, one defines the least upper bound $u \vee^{\mathcal{U}} \tilde{u}$. It follows that for all $u, \tilde{u} \in \mathcal{U}$,
	\begin{align*}
		& \mu^{u \wedge^{\mathcal{U}} \tilde{u}} = \mu^{u} \wedge^{L} \mu^{\tilde{u}}, & & \mu^{u \vee^{\mathcal{U}} \tilde{u}} = \mu^{u} \vee^{L} \mu^{\tilde{u}}. &
	\end{align*}
	
	Let $(u_{n})_{n\in \mathbb{N}} \subseteq \mathcal{U}$. Set $\hat{\mu} := \inf_{n\in \mathbb{N}} \mu^{u_{n}}$, and define $\hat{u} \in \mathcal{U}$ by setting $\hat{u}(0) := \inf_{n\in \mathbb{N}} u_{n}(0)$,
	\[
		\hat{u}(t+1) := \inf\left\{ u_{n}(t)\cdot \frac{p+(\mu^{u_{n}}_{t})_{2}(q-p)}{p+(\hat{\mu}_{t})_{2}(q-p)} : n\in \mathbb{N} \right\}, \quad t\in \{0,\ldots,T-2\}.
	\]
	By induction, one checks that $\mu^{\hat{u}} = \hat{\mu}$, hence the part of Condition \ref{FiniteMarkovSufficient.compatibility} in Assumption \ref{FiniteMarkovSufficient} regarding the greatest lower bound is satisfied. 
	The upper bound part is analogous.
	
	Regarding the costs, choose zero running costs $f\equiv 0$ and terminal costs $g$ given by
	\[
		g(i,m) := \phi(i)\cdot \psi(m_{2}),\quad i \in \{1,2\},
	\]
	where $\phi(2) < \phi(1)$ and $\psi\colon [0,1] \rightarrow \mathbb{R}$ is nondecreasing (but not necessarily continuous).
	Then, for $u\in \mathcal{U}$, $\mu \in L$,
	\[
		J\left((u,\mu^{u}), \mu \right) = \left( \bigl(\phi(2)-\phi(1) \bigr) (\mu^{u}_{T})_{2} + \phi(1) \right)\cdot \psi\left((\mu_{T})_{2}\right).
	\]
	Here, if $(\mu^{(n)})_{n\in \mathbb{N}} \subset L$ and $\hat{\mu} = \inf_{n\in \mathbb{N}} \mu^{(n)}$, $\check{\mu} = \sup_{n\in \mathbb{N}} \mu^{(n)}$, then
	\begin{align*}
		& (\hat{\mu}_{T})_{2} = \inf \left\{ (\mu^{(n)}_{T})_{2} : n\in \mathbb{N} \right\}, & & (\check{\mu}_{T})_{2} = \sup \left\{ (\mu^{(n)}_{T})_{2} : n\in \mathbb{N} \right\}. &
	\end{align*}
	The form of $J$ and Condition \ref{FiniteMarkovSufficient.compatibility} in Assumption \ref{FiniteMarkovSufficient} therefore imply that Condition \ref{FiniteMarkovSufficient.costinfsup} in Assumption \ref{FiniteMarkovSufficient} holds. 
	
	In order to check the submodularity condition (i.e., Condition \ref{FiniteMarkovSufficient.submod} in Assumption \ref{FiniteMarkovSufficient}), let $\hat{\mu}, \check{\mu}\in L$ and $\hat{u}, \check{u} \in \mathcal{U}$ be such that $\hat{\mu} \leq^{L} \check{\mu}$, $\hat{u} \leq^{\mathcal{U}} \check{u}$. Then
	\begin{align*}
		J\left( (\check{u}, \mu^{\check{u}}), \check\mu \right) - J\left( (\hat{u}, \mu^{\hat{u}}), \check\mu \right) &= \left( \bigl(\phi(2)-\phi(1) \bigr) \left( (\mu^{\check{u}}_{T})_{2} -  (\mu^{\hat{u}}_{T})_{2} \right) \right)\cdot \psi\left((\check{\mu}_{T})_{2} \right), \\
		J\left( (\check{u}, \mu^{\check u}), \hat{\mu} \right) - J\left( (\hat{u}, \mu^{\hat{u}}), \hat{\mu} \right) &= \left( \bigl(\phi(2)-\phi(1) \bigr) \left( (\mu^{\check{u}}_{T})_{2} -  (\mu^{\hat{u}}_{T})_{2} \right) \right)\cdot \psi\left((\hat{\mu}_{T})_{2} \right).
	\end{align*}
	But $\phi(2)-\phi(1) < 0$, while $(\mu^{\check{u}}_{T})_{2} - (\mu^{\hat{u}}_{T})_{2} \geq 0$ by Condition \ref{FiniteMarkovSufficient.compatibility} in Assumption \ref{FiniteMarkovSufficient} since $\hat{u} \leq^{\mathcal{U}} \check{u}$, and $\psi((\hat{\mu}_{T})_{2}) \leq \psi((\check{\mu}_{T})_{2})$ since $\hat{\mu} \leq^{L} \check{\mu}$ and $\psi$ is nondecreasing. It follows that
	\[
		J\left( (\check{u}, \mu^{\check{u}}), \check\mu \right) - J\left( (\hat{u}, \mu^{\hat{u}}), \check\mu \right) \leq J\left( (\check{u}, \mu^{\check u}), \hat{\mu} \right) - J\left( (\hat{u}, \mu^{\hat{u}}), \hat{\mu} \right),
	\]
	which is Condition \ref{FiniteMarkovSufficient.submod} in Assumption \ref{FiniteMarkovSufficient}.
	
\end{example}

\section{Submodular mean field games with singular controls}
\label{section singular controls}

In this section, we specialize to mean field games with singular controls, and show that they can be embedded into the general set-up given in Section \ref{section general case}. 
In the following, we consider MFGs with common noise in which the representative player faces a convex optimization problem (see Subsection \ref{section geometric BM} below) and MFGs without common noise, in which the representative player faces a nonconvex optimization problem (see Subsection \ref{section nonconvex singular} below). 
In these two models, the operations, which are postulated in the Assumption \ref{assumption}, can be constructed with different techniques. 
These operations can be explicitly constructed in the case in which the dynamics are given by controlled geometric Brownian motions and the costs are convex in the state variable.   
When the dynamics are nonlinear, the construction of such operations is provided by approximating singular controls via regular controls, and exploiting the results in \cite{dianetti.ferrari.fischer.nendel.2019}.    

Throughout this section, we take measurable functions
\begin{align*}
 f &\colon [0,T] \times \mathbb{R} \times \mathcal{P}(\mathbb{R}) \to  \R, \\
 g &\colon  \mathbb{R} \times \mathcal{P}(\mathbb{R}) \to  \R, \\
 c &\colon [0,T] \to [0,\infty), 
\end{align*}
satisfying the following conditions.
\begin{assumption} 
\label{ass.submodularity.continuos.time setting} \
\begin{enumerate}
\item For  $ dt$-a.a.\ $t \in  [0,T]$, the functions $f(t,\cdot, \mu)$ and $g(\cdot,\mu)$ are lower semicontinuous and, for some $p>1$ and all $ (t,x,\mu) \in [0,T] \times \mathbb{R} \times \mathcal{P}(\mathbb{R})$,
    $$
    \kappa (|x|^p-1) \leq f(t, x, \mu) \leq K(1+|x|^p), \quad \kappa( |x|^p -1) \leq g(x,\mu) \leq K(1+|x|^p),
    $$
    with constants $K,\kappa >0$;
    \item  For  $ dt$-a.a.\ $t \in  [0,T]$, the functions $f(t,\cdot, \cdot)$ and $g$  have decreasing differences in $(x,\mu)$; that is, for $\phi \in \{f(t,\cdot, \cdot), g \}$,
\begin{equation*}  
 \phi(\bar{x},\bar{\mu})- \phi(x,\bar{\mu}) \leq \phi(\bar{x},\mu)-  \phi(x,\mu ), 
\end{equation*}
for all $ \bar{x},x \in \mathbb{R}$ and $ \bar{\mu},\mu \in \mathcal{P}(\mathbb{R})$ with $\bar{x} \geq x$ and $\bar{\mu} \geq^{\text{st}} \mu$.
\item The cost $c$ is nonincreasing and continuously differentiable with $c > 0$.
\end{enumerate}
\end{assumption}
\subsection{Controlled geometric Brownian motion and common noise}
\label{section geometric BM} 
\subsubsection{Formulation of the model}
Let Assumption \ref{ass.submodularity.continuos.time setting} be satisfied with $p=2$.
Let $W=(W_t)_{t\in [0,T]}$ and $B=(B_t)_{t\in [0,T]}$  be two independent  Brownian motions on a complete filtered probability space $(\Omega,\mathcal F,\mathbb F,\mathbb P)$. 
Define the set of \emph{admissible monotone controls} as the set $\mathcal{V}_{\uparrow}$ of all $\mathbb{F}$-adapted c\`adl\`ag, nondecreasing, square-integrable, and nonnegative processes $\xi=(\xi_t)_{t \in [0,T]}$ such that
\begin{equation}\label{gbm measure pi}
 \mathbb E\bigg[  \int_0^T \xi_t^2  d \pi (t) \bigg] < \infty, \text{ where $\pi:= dt + \delta _T$.}    
\end{equation}
Let $b \in \mathbb{R}$, $\sigma, \, \sigma^o \geq 0$, and $\mathbb F ^o:=(\mathcal{F}_t^o )_{t\in[0,T]}$ denote the filtration generated by $\sigma^o B$ (which is trivial in the case of no common noise, i.e., for $\sigma^o=0$). 
Let $x_0$ be a square integrable $\mathcal{F}_0$-random variable.
For each $\xi \in \mathcal{V}_{\uparrow}$, let $X^{\xi}=(X_t^{\xi})_{t\in[0,T]}$ denote the unique strong solution to the linearly controlled geometric dynamics, given by
\begin{equation}\label{SDE.singular.geometric} 
dX_t^{\xi} =  X_t^\xi (b\, dt + \sigma dW_t + \sigma^o dB_t ) + d\xi_t, \quad t\in [0,T],  \quad X_{0-}^\xi = x_0.  
\end{equation}

For any $\mathcal{P}(\mathbb R)$-valued ${\mathbb{F}}^o$-progressively measurable process $\mu=\left( \mu_t \right)_{t \in [0,T]}$, we
introduce the cost functional
\begin{equation*}
 J(\xi, \mu):= \mathbb{E} \bigg[ \int_0^T f(t, X_t^{\xi}, \mu_t) dt  + g( X_T^{\xi},\mu_T ) + \int_{[0,T]} c_t d\xi_t  \bigg] , \quad \xi \in \mathcal{V}_{\uparrow},
\end{equation*} 
and consider the  singular control problem $\inf_{\xi \in \mathcal{V}_{\uparrow}} J( \xi, \mu )$.
We say that $(X^{\mu},\xi^{\mu})$ is an \emph{optimal pair} for the flow $\mu$ if $ J( \xi^\mu, \mu) \leq J(\xi, \mu) $ for each admissible $\xi$ and $X^\mu=X^{\xi^\mu}$. 
\begin{definition}
 A $\mathcal{P}(\mathbb R)$-valued ${\mathbb{F}}^o$-progressively measurable process $\mu=\left( \mu_t \right)_{t \in [0,T]}$ is an equilibrium of the MFG with singular controls and common noise if
\begin{enumerate}
    \item there exists an optimal pair $(X^{\mu},\xi^{\mu})$ for $\mu$,
    \item $\mu_t=\mathbb P [ X_t^\mu \in \cdot \, | \mathcal F_T^o ]$ $\mathbb P$-a.s., for any $t \in[0,T]$.
\end{enumerate}
\end{definition}

\subsubsection{Optimal controls and a priori estimates}
Recalling that $c >0$, we enforce the following requirements.
\begin{assumption}
\label{assumption singular linear convex case} \
 For  $ dt$-a.a.\ $t \in  [0,T]$, the functions $f(t,\cdot, \mu)$ and $g(\cdot,\mu)$ are strictly convex.
\end{assumption}

Under Assumption  \ref{assumption singular linear convex case}, by employing arguments as those in the proof of Theorem 8 in \cite{Menaldi&Taksar89}, it can be shown that for any process $\mu$, there exists a unique optimal pair $(X^{\mu},\xi^{\mu})$. 
Moreover, since the control, which constantly equals to 0, is suboptimal, the growth conditions in Assumption \ref{assumption singular linear convex case} imply that
\begin{align*}
\kappa \mathbb E\bigg[ \int_0^T |X_t^\mu|^2 dt + |X_T^\mu|^2 \bigg] -\kappa(1+T) & \leq J(\xi^\mu,\mu) \\ 
&\leq J(0,\mu)   \leq K \mathbb  E\bigg[ \int_0^T |X_t^0|^2 dt + |X_T^0|^2 \bigg] + K (1+T), 
\end{align*}
so that, for some constant $\bar{C}>0$ independent of $\mu$, we have
\begin{equation*}
\mathbb E\bigg[ \int_0^T |X_t^\mu|^2 dt + |X_T^\mu|^2 \bigg] \leq \bar{C}.   
\end{equation*}
Therefore, for a suitable generic constant $C>0$ (changing from line to line), we obtain
\begin{align*}
\mathbb{E} [|\xi_T^\mu |^2]
& \leq \mathbb{E} \bigg[ \bigg( X_T^\mu - x_0 - \int_0^T X_t^\mu (b\, dt + \sigma dW_t + \sigma^o dB_t) \bigg)^2 \bigg]  \\ 
&\leq C \mathbb{E} \bigg[ |X_T^\mu|^2  +|x_0|^2 + \int_0^T|X_t^\mu|^2 dt \bigg] \leq C,    
\end{align*}
and, by a standard use of Gr\"onwall's inequality, we conclude that
\begin{equation}\label{tightness.singular.geometric}
     \mathbb{E}[|X_t^\mu |^2 +|\xi_t^\mu|^2] \leq M, \quad \text{for each } t\in [0,T],
 \end{equation}
for a constant $M>0$, which does not depend on $\mu$.

\subsubsection{The control set $E$ and its operations} Define 
\begin{align*}
E:=\big\{ (X^\xi,\xi) \,|\, &\xi \in \mathcal{V}_\uparrow, \text{ $X^\xi$ solution to \eqref{SDE.singular.geometric}} \big\} \quad \text{and}\\ p(X^\xi,\xi)_t (A)&:= \mathbb{P} [ X_t^\xi\in A \, | \mathcal{F}_T^o ], \quad A \in \mathcal{B}(\R).  
\end{align*} 
Due to \eqref{gbm measure pi}, the set $E$ is a subset of the space $\mathbb L^2 _{\pi}$ of $\R^2$-valued progressively measurable processes $\nu$ such that $ \| \nu \|_{\pi,2} := \mathbb E\big[ \int_0^T |\nu_t|^2 d \pi (t) \big] < \infty$, endowed with the norm $\| \cdot \|_{\pi,2}$. 
Moreover, the lower semicontinuity properties of $J$ in Assumptions \ref{assumption.bestresponsemap.nonempty} and \ref{assumption.general.approximation} are satisfied, while the continuity of $J$ w.r.t.\ $\mu$ holds by assuming $f$ and $g$ to be continuous in $\mu$.

Observe that, for each $\xi \in \mathcal{V}_\uparrow$, the solution to the SDE \eqref{SDE.singular.geometric} is, $\P$-a.s., given by
 \begin{equation}
     \label{eq SDE explicit geometric}
     X_t^\xi = \mathcal{E}_t \bigg[ x_0 + \int_{[0,t]} \mathcal{E}_s^{-1} d\xi_s \bigg]\quad  \text{with}\quad 
     \mathcal{E}_t:= \exp \Big[ \Big(b-\tfrac{(\sigma^2 +(\sigma ^o)^2)}{2}\Big)t +  \sigma  W _t + \sigma^o B_t\Big]
 \end{equation} 
 for each $t\in [0,T]$.
 Hence, defining the map $\Phi\colon\mathcal{V}_\uparrow \to \mathcal{V}_\uparrow$ by $\Phi_t(\xi):= \int_{[0,t]} \mathcal{E}_s^{-1} \,d\xi_s$, we have, $\P$-a.s.,
 $$
 X_t^\xi = \mathcal{E}_t[x_0 + \Phi_t(\xi)], \quad  \text{for each } t\in [0,T].
 $$
Moreover, for $\bar{\xi}, \xi \in \mathcal{V}_\uparrow$ and $\bar{\zeta}:=\Phi(\bar{\xi})$ and ${\zeta}:=\Phi({\xi})$, we define, $\P$-a.s., the  controls
\begin{equation}
\label{eq def xi land lor}
\xi_t^\land:=\int_{[0,t]} \mathcal{E}_s d(\bar{\zeta} \land \zeta)_s \quad \text{and} \quad
\xi_t^\lor:=\int_{[0,t]} \mathcal{E}_s d(\bar{\zeta} \lor \zeta)_s, \quad \text{for each } t\in [0,T],
\end{equation}
and obtain
\begin{align}
\label{eq sup dinamics dinamics equals controlled by the sup}
 & X_t^{\bar{\xi}} \land X_t^\xi = \mathcal{E}_t[x_0 + \bar{\zeta}_t \land \zeta_t ]=  \mathcal{E}_t \bigg[ x_0 + \int_{[0,t]} \mathcal{E}_s^{-1} d\xi_s^\land \bigg] =X_t^{\xi^\land}\quad\text{and}  \\ \notag
 & X_t^{\bar{\xi}} \lor X_t^\xi = \mathcal{E}_t[x_0 + \bar{\zeta}_t \lor \zeta_t ]=  \mathcal{E}_t \bigg[ x_0 + \int_{[0,t]} \mathcal{E}_s^{-1} d\xi_s^\lor \bigg] =X_t^{\xi^\lor}. 
\end{align}

According to \eqref{eq def xi land lor}, we introduce the operations $\land^{\text{\tiny{$E$}}},\lor^{\text{\tiny{$E$}}} \colon E \times E \to E$ via 
\begin{equation}\label{equation singular definition operation linear}
(X^{\bar{\xi}},\bar{\xi}) \land^{\text{\tiny{$E$}}} (X^\xi,\xi) := (X^{\xi^\land}, \xi^\land) \quad \text{ and }  \quad (X^{\bar{\xi}},\bar{\xi}) \lor^{\text{\tiny{$E$}}} (X^\xi,\xi) := (X^{\xi^\lor}, \xi^\lor).
\end{equation}
Note that, in light of \eqref{eq sup dinamics dinamics equals controlled by the sup}, the operations $\land^{\text{\tiny{$E$}}},\lor^{\text{\tiny{$E$}}}$ satisfy Condition \ref{assumption.existence.operation} in Assumption \ref{assumption}. 

\subsubsection{The submodularity condition} 
Using the definition of $\xi^\lor$, the linearity of the integral, and that $ \bar{\zeta} \lor \zeta - \bar{\zeta} = {\zeta}- \bar{\zeta}\land \zeta$, we obtain that, for each $t\in[0,T]$,
\begin{equation}\label{gbm eq submodularity of control}
\xi_{t}^\lor - \bar{\xi}_{t}  = \int_{[0,t]} \mathcal{E}_s (d(\bar{\zeta}\lor \zeta)_s-d\bar{\zeta}_s) = \int_{[0,t]} \mathcal{E}_s (d{\zeta}_s-d(\bar{\zeta}\land \zeta)_s) = \xi_{t} - {\xi}_{t}^\land\quad \P\text{-a.s.}
\end{equation}

Recalling the definition of the measure $\pi$ in \eqref{gbm measure pi}, for  $\mu,  \bar{\mu} \in L$, we define the order relation 
\begin{equation}\label{gbm def order relation}
\mu \leq^{\text{\tiny{$L$}}} \bar \mu \text{ if and only if $\mu_t \leq \nu_t, \ \mathbb P$-a.s., for $\pi$-a.a.\ $t\in [0,T]$}.
\end{equation} 
Now, let $\mu, \bar \mu$ be two $\mathcal{P}(\mathbb R)$-valued, $\mathbb F^o$-progressively measurable processes with $\mu \leq^{\text{\tiny{$L$}}} \bar \mu$ and $\xi, \bar \xi \in \mathcal{V}_\uparrow$. 
Using \eqref{eq sup dinamics dinamics equals controlled by the sup} and \eqref{gbm eq submodularity of control}, we find
\begin{align}\label{gbm eq submodularity J} 
 J(\xi^\lor, \bar \mu)- J(\bar{\xi}, \bar \mu) 
& = \mathbb{E} \bigg[ \int_0^T (f(t, X_t^{\bar \xi} \lor X_t^{ \xi} , \bar{\mu}_t) - f(t, X_t^{\bar \xi}, \bar{\mu}_t)) dt \bigg] \\ \notag  
& \quad + \E \bigg[ g( X^{\bar \xi}_T \lor X^{ \xi}_T,\bar{\mu}_T ) - g( X^{\bar \xi}_T,\bar{\mu}_T ) 
+ \int_{[0,T]} c_t d(\xi^\lor - \bar{\xi})_t  \bigg] \\ \notag
& = \mathbb{E} \bigg[ \int_0^T (f(t,  X_t^{ \xi} , \bar{\mu}_t) - f(t, X_t^{\bar \xi} \land X_t^{ \xi}, \bar{\mu}_t)) dt \bigg] \\ \notag
& \quad + \E \bigg[ g( X^{ \xi}_T ,\bar{\mu}_T ) - g( X_T^{\bar \xi} \land X_T^{ \xi},\bar{\mu}_T ) 
+ \int_{[0,T]} c_t d(\xi - {\xi}^\land)_t  \bigg] \\ \notag
 & = J(\xi, \bar \mu)- J(\xi^\land, \bar \mu). \notag
\end{align}
Moreover, by using \eqref{eq sup dinamics dinamics equals controlled by the sup} and Assumption \ref{ass.submodularity.continuos.time setting}, we obtain that
\begin{align}\label{gbm eq decreasing differences J}
J(\xi, \bar \mu)- J({\xi}^\land, \bar \mu) & \leq  \mathbb{E} \bigg[ \int_0^T (f(t,  X_t^{ \xi} , \mu_t) - f(t, X_t^{\bar \xi} \land X_t^{ \xi}, \mu_t)) dt \bigg] \\ \notag
& \quad + \E \bigg[ g( X^{ \xi}_T ,\mu_T ) - g( X_T^{\bar \xi} \land X_T^{ \xi},\mu_T ) 
+ \int_{[0,T]} c_t d(\xi - {\xi}^\land)_t  \bigg] \\ \notag 
&=J(\xi,  \mu)- J({\xi}^\land,  \mu). \notag
\end{align}
Note that \eqref{gbm eq submodularity J} and \eqref{gbm eq decreasing differences J} imply that Condition \ref{assumption.submod.J} in Assumption \ref{assumption} is satisfied, so that the operations $\land^{\text{\tiny{$E$}}},\lor^{\text{\tiny{$E$}}}$ verify all the requirements of Assumption \ref{assumption}.  

Moreover, taking $\xi \in \argmin_{\mathcal{V}_\uparrow} J(\cdot, \mu)$ and  $\bar \xi \in \argmin_{\mathcal{V}_\uparrow} J(\cdot, \bar \mu)$, and using \eqref{gbm eq submodularity J} and \eqref{gbm eq decreasing differences J} we find that $\xi^\land \in \argmin_{\mathcal{V}_\uparrow} J(\cdot, \mu)$ and $ \xi^\lor \in \argmin_{\mathcal{V}_\uparrow} J(\cdot, \bar \mu)$.
Therefore, by the uniqueness of optimal controls, we conclude that $\xi^\land = \xi$ and $\xi^\lor = \bar{\xi}$, so that
\begin{equation}\label{eq mfgsing pre best reply map increasing}
X_t^\mu \leq X_t^{\bar{\mu}}, \ \pi \text{-a.e., whenever $\mu \leq^{\text{\tiny{$L$}}} \bar \mu$.} 
\end{equation}
 
\subsubsection{The lattice $L$} We move on to the identification of a suitable partially ordered set ${(L, \leq^{\text{\tiny{$L$}}})}$.  
Thanks to the a priori estimate \eqref{tightness.singular.geometric} and Chebyshev's inequality for conditional probabilities, we obtain (employing the convention ${x}/{0}= \infty$ for any $x \geq 0$)
\begin{align}\label{eq sing max lemma 1}
\mathbb P [ X_t^\mu \leq x | \mathcal{F}_T^o ] 
&\geq \bigg( 1 - \frac{\mathbb E [|X_t^\mu|^2| \mathcal{F}_T^o]}{(x \lor 0)^2} \bigg) \lor 0 
\geq \bigg( 1 - \frac{\esssup_\mu  \mathbb E [|X_t^\mu|^2| \mathcal{F}_T^o]}{(x \lor 0)^2} \bigg) \lor 0\\ \notag
&=:\mu_t^{\rm Max} \big((-\infty,x]\big), \notag 
\end{align}
as well as
\begin{equation}\label{eq sing max lemma 2}
  \mathbb P [ X_t^\mu \leq x | \mathcal{F}_T^o ] \leq \frac{\mathbb E [ |X_t^\mu|^2 | \mathcal{F}_T^o] }{(x \land 0)^2} \land 1 \leq  \frac{ \esssup_{\mu} \mathbb E [ |X_t^\mu|^2 | \mathcal{F}_T^o] }{(x \land 0)^2} \land 1=:\mu_t^{\rm Min} \big((-\infty,x]\big)   
\end{equation}
for any $\mathcal P(\R)$-valued $\mathbb F^o$ -progressively\ measurable\ flow $\mu$. From \eqref{eq mfgsing pre best reply map increasing}, we see that the set $\{ X^\mu  \, | \, \mu \text{ is a $\mathcal P(\R)$-valued $\mathbb F^o$ -progr.\ meas.\ flow}\}$ is directed downwards (and upwards). 
Therefore, by the monotone convergence theorem,
$$\esssup_{\mu} \mathbb E [ |X_t^\mu|^2 | \mathcal{F}_T^o] \in \mathbb L^1 (\Omega; \mathbb P),$$
so that $\esssup_{\mu} \mathbb E [ |X_t^\mu|^2 | \mathcal{F}_T^o] < \infty$ $\mathbb P$-a.s.\ We deduce that the $\mathbb F^o$-progressively measurable processes $\mu^{\rm Min}$ and $\mu^{\rm Max}$ are $\mathcal{P}(\mathbb R)$-valued and that, for all $\mathcal P(\R)$-valued $\mathbb F^o$ -progressively\ measurable\ flows $\mu$,
\begin{equation}\label{gbm estimate optimal trajectories}
\mu_t^{\rm Min} \leq^{\text{st}} 
\mathbb{P} [ X_t^\mu \in \cdot \; | \mathcal{F}_T^o ] 
\leq^{\text{st}} \mu_t^{\rm Max} \quad \text{$\mathbb P$-a.s.,\quad for all } t \in [0,T].
\end{equation}
We therefore consider the set $L$ of all $\mathcal{P}(\mathbb R)$-valued, $\mathbb F^o$-progressively measurable\ processes $\mu$ with $$\mu_t^{\rm Min}\leq \mu_t \leq \mu_t^{\rm Max} \quad\P\text{-a.s.,\quad for $\pi$-a.a.\ $t \in [0,T]$}, 
$$ 
endowed with the order relation $\leq^{\text{\tiny{$L$}}}$ defined in \eqref{gbm def order relation}.
Since the ordered set $(L, \leq^{\text{\tiny{$L$}}})$ is a special instance of the lattice $\mathcal L$ considered in Section \ref{section lattices}, which is, in addition, order-bounded, it is a complete and Dedekind super complete lattice. 

\subsubsection{Existence and approximation of equilibria}
For any $\mu \in L$, set $ R(\mu)_t := \mathbb{P} [ X_t^\mu \in \cdot \, | \mathcal{F}_T^o ],$ for each $t\in [0,T]$.   
Thanks to \eqref{gbm estimate optimal trajectories}, the best-reply-map $R\colon L \to L$ is well-defined and the MFG equilibria of the MFG with singular controls correspond to processes $\mu \in L$ with $R(\mu)=\mu$.

We can now state and prove the main result of this subsection.
\begin{theorem}
The set of solutions of the MFG with singular controls and common noise is a nonempty complete lattice. Moreover, if $f$ and $g$ are continuous in $(x,\mu)$, then
\begin{enumerate}
\item\label{gbm learning 1} the learning procedure $\underline{\mu}^n$ defined inductively by $\underline{\mu}^0=\inf L$ and $\underline{\mu}^{n+1}=R(\underline{\mu}^n)$ is nondecreasing in $L$ and it converges to the minimum MFG solution,
\item\label{gbm learning 2} the learning procedure $\overline{\mu}^n$ defined inductively by $\overline{\mu}^0=\sup L$ and $\overline{\mu}^{n+1}=R(\overline{\mu}^n)$ is nonincreasing in $L$ and it converges to the maximal MFG solution.  
\end{enumerate}
\end{theorem}
\begin{proof}
The fact that the set of MFG solutions is a nonempty complete lattice is a direct consequence of Theorem \ref{theorem.main.general}. 
We therefore just prove the convergence of the learning procedure in Claim \ref{gbm learning 1} (Claim \ref{gbm learning 2} can be proved analogously). 
Even if the sequential compactness in Assumption \ref{assumption.general.approximation} is not satisfied, the arguments in the proof of Theorem \ref{theorem.main.general} can be recovered as follows. 

We first observe that, thanks to \eqref{eq mfgsing pre best reply map increasing} and the definition of $R$, the sequence $\underline{\mu}^n$ is nondecreasing in $L$.
Hence, setting $(X^n, \xi^n) :=( X^{\underline{\mu}^n}, \xi^{\underline{\mu}^n})$, again by \eqref{eq mfgsing pre best reply map increasing} we have that $X_t^n \leq X_t ^{n+1}$,  $\mathbb P \otimes \pi$-a.e., for any $n\in \mathbb N$. 
Therefore, we can define the process $X_t := \sup_n X_t^n$, and, by the monotone convergence theorem and the estimates in \eqref{tightness.singular.geometric}, we conclude that $X^n \to X$ in $\mathbb L^2 _{\pi}$ as $n\to \infty$.
Next, we define the control process $\xi$ by setting 
$$
\xi_t := X_t - x_0 - \int_0^t X_s(b dt + \sigma dW_s + \sigma^o dB_s). 
$$
The convergence of $X^n$ in $\mathbb L^2 _{\pi}$ implies that $\xi^n \to \xi$ in $\mathbb L^2 _{\pi}$ as $n\to \infty$, so that $\xi$ is nondecreasing. 
Employing Lemma 3.5 in \cite{K}, we can take c\`adl\`ag versions of $X$ and $\xi$, so that $(X,\xi) \in E$. After repeating the arguments from the proof of Theorem \ref{theorem.main.general}, the proof is complete. 
\end{proof}

\subsection{Nonconvex case without common noise}\label{section nonconvex singular}
In this subsection, we treat a model of mean field games with singular controls and no common noise, for a general drift and a not necessarily convex running cost.
As a consequence, optimal controls are in general not unique. 
In comparison with the previous subsection, this case requires a more technical analysis, which makes use of a weak formulation of the problem in the spirit of \cite{Haussmann&Suo95}. 

\subsubsection{Model formulation}
Let $\sigma \geq 0$ be a constant and $b\colon [0,T] \times \R \to \R$ be a Lipschitz continuous function.\
In order to come up with a weak formulation of the problem, the initial value of the dynamics will be described through a fixed initial distribution $\nu_0 \in \mathcal{P}(\mathbb{R)}$, satisfying $|\nu_0|^{p}:=\int_{\mathbb{R}} |y|^{p} d\nu_0 (y) < \infty$ with $p>1$ from Assumption \ref{ass.submodularity.continuos.time setting}. 
 
\begin{definition}
\label{def.admissible.singular.control}
A tuple $\rho=(\Omega,\mathcal{F},\mathbb{F}, \mathbb{P}, x_0 , W, \xi)$ is said to be an admissible singular control if
\begin{enumerate}
\item $(\Omega,\mathcal{F},\mathbb{F}, \mathbb{P})$ is a filtered probability space satisfying the usual conditions;
\item $x_0$ is an $\mathcal{F}_0$-measurable $\mathbb{R}$-valued random variable with $\mathbb{P}\circ x_0^{-1} = \nu_0$; 
\item $W$ is a standard $(\Omega,\mathcal{F},\mathbb{F}, \mathbb{P})$-Brownian motion;
\item $\xi \colon  \Omega \times [0,T] \to [0,\infty)$ is an $\mathbb{F}$-adapted nondecreasing c\`adl\`ag process.
\end{enumerate}
We denote by $E^w$ the set of admissible singular controls.  
\end{definition}
 
Again, since $b$ is assumed to satisfy the usual Lipschitz continuity and growth conditions, for any $\rho \in E^w$  there exists a unique process $X^\rho:\Omega \times [0,T] \to \mathbb{R}$ solving the system's dynamics equation that now reads as
\begin{equation}
 \label{dynamics.singular.controls}
 X_t^\rho = x_0 + \int_0^t  b(t,X_t^\rho) dt + \sigma W_t + \xi_t, \quad t \in [0,T].  
\end{equation}  
Then, for a measurable flow of probability measures $\mu$, we define the cost functional
\[
J (\rho,\mu) := \mathbb{E}^{\mathbb{P}}\bigg[ \int_0^T  f(t,X_t^\rho,\mu_t) dt  + g(X_T^\rho,\mu_T) + \int_{[0,T]} c_t d\xi_t \bigg], \quad  \rho \in E^w,  
\]
and we say that $\rho \in E^w $ is an \emph{optimal control} for the flow of measures $\mu$ if it solves the optimal control problem related to $\mu$; that is, if $ {J}(\rho,\mu) =\inf_{E^w} {J}(\cdot,\mu)$.

\begin{definition} A measurable flow of probabilities ${\mu}$ is a MFG equilibrium if
\begin{enumerate} 
    \item there exists an optimal control $\rho \in E^w $  for $\mu$,
    \item ${\mu}_t= \mathbb{P} \circ (X_t^{{\rho}})^{-1}$ for any $t\in [0,T]$. 
\end{enumerate} 
\end{definition} 

\subsubsection{Reformulation via control rules and preliminary remarks}
In order to have a topology on the space of admissible controls, we reformulate the problem in terms of control rules. 
We introduce the following canonical space $(\Omega, \mathcal F)$ by
\begin{equation}\label{canonical space}  
\Omega := \R \times \mathcal C \times \mathcal{D} \times \mathcal{D}_{\uparrow}, \quad
\mathcal{F} := \mathcal B (\R) \otimes \mathcal{B}(\mathcal{C}) \otimes \mathcal{B}(\mathcal{D}) \otimes \mathcal{B}(\mathcal{D}_{\uparrow}).
\end{equation}   
We define the set of control rules as
$$
E:= \{ \nu^\rho\, |\,  \rho \in E^w \}, \quad \text{where } \nu^\rho:= \mathbb P \circ (x_0, W, X^\rho, \xi)^{-1} \text{ for }\rho = ({\Omega}, {\mathcal{F}}, {\mathbb{F}}, {\mathbb{P}}, x_0 ,W, \xi) \in E^w,
$$
and, with a slight abuse of notation, we set $J(\nu^\rho,\mu):= J(\rho, \mu)$. 
In this way, $E$ is naturally defined as a subspace of the topological space $\mathbb P (\Omega)$.  

\begin{remark}[Existence of optimal controls]\label{remark singular control BRM nonempty}
Under the standing assumptions, it is shown in \cite{Haussmann&Suo95} that, for each measurable flow of probabilities $\mu$, $J(\cdot, \mu)$ is lower semicontinuous and the set 
$\argmin_{E} J(\cdot, \mu) \subset E$ is nonempty (see Theorem 3.6 and Theorem 3.8 in \cite{Haussmann&Suo95}).    
Also, one can show that (see Theorem 3.7 in \cite{Haussmann&Suo95}), for each sequence  
$
( \nu_n )_{ n \in \mathbb{N}} \subset \argmin_{E} J(\cdot, \mu),
$ 
we can find an admissible singular control 
$
\nu \in \argmin_{E} J(\cdot, \mu)
$
such that, up to a subsequence, $\nu_n$ converges weakly to $\nu$ in $\mathcal P (\Omega)$.   
\end{remark} 

Now, for any measurable flow of measures $\mu$, if $\rho \in \argmin _{E^w} J(\cdot, \mu)$,  we can repeat (with minor modifications) the arguments leading to \eqref{tightness.singular.geometric} in order to get a priori estimates on the moments of optimally controlled trajectories;
namely, we have
\begin{equation}\label{singular nonconvex a priori estimates}  
\mathbb{E}^{{\mathbb{P}}} [|X_t^{\rho}|^p + (\xi_T)^p] \leq M, \quad  \text{for any } t \in [0,T]\text{ and }\rho \in \argmin _{E^w} J(\cdot, \mu),  
\end{equation}
with a constant $M>0$ independent of the flow of measures $\mu$.
Therefore, following computations similar to those leading to \eqref{eq sing max lemma 1} and \eqref{eq sing max lemma 2} (see also Lemma 3.4 in \cite{nendel}), we can find $\mu^{\rm Min},\, \mu^{\rm Max}\in \mathcal P(\R)$ such that, for any flow of measures $\mu$, one has 
\begin{equation}\label{sub mfg sing a priori estimate equation minimizers}
\mu^{\rm Min} \leq^{\text{st}} \mathbb P \circ (X_t^\rho)^{-1} \leq^{\text{st}} \mu^{\rm Max}, \quad  \text{for any } t \in [0,T] \text{ and } \rho \in \argmin _{E^w} J(\cdot, \mu). 
\end{equation}
 
We thus define the set of feasible flows of measures $L$ 
as the set of all equivalence classes (w.r.t.\ the measure $\pi:= dt+\delta_T$ on the interval $[0, T]$) of measurable flows of probabilities $\mu \colon [0, T] \to \mathcal P( \mathbb R)$ with $\mu_t \in [\mu^{\rm Min}, \mu^{\rm Max}]$ for $\pi$-a.a.\ $t \in [0,T]$.  
On $L$ we consider the order relation $\leq^{\text{\tiny{$L$}}}$ given by $\mu \leq^{\text{\tiny{$L$}}} \nu$ if and only if $\mu_t \leq^{\text{st}} \bar{\mu}_t$, for $\pi$-a.a.\ $t\in [0,T]$, with the lattice structure given by
\[
(\mu \land^{\text{\tiny{$L$}}} \bar{\mu})_t:=\mu_t \land^{\text{st}} \bar{\mu}_t \quad \text{and} \quad 
(\mu \lor^{\text{\tiny{$L$}}} \bar{\mu})_t:=\mu_t \lor^{\text{st}} \bar{\mu}_t \quad \text{for } \pi\text{-a.a. } t  \in [0,T].
\] 
Again, this is a particular instance of the lattice $\mathcal L$ considered in Section \ref{section lattices}, and it is, by definition, norm-bounded, As a consequence, $(L,\leq^{\text{\tiny{$L$}}})$ is a complete and Dedekind super complete lattice.
 
Next, we can define the set 
$$
\text{$E^{\text{\tiny{$M$}},w} := \{ \nu \in E^w \, |\,$\eqref{singular nonconvex a priori estimates} holds$\}$\quad and\quad$E^{\text{\tiny{$M$}}}:= \{ \nu^\rho \,|\, \rho \in E^{\text{\tiny{$M$}},w} \}$,} 
$$
so that $\argmin_E J(\cdot, \mu) \subset E^{\text{\tiny{$M$}}}$ for any flow $\mu$.
We observe that, due to the Meyer-Zheng tightness criteria (see Theorem 4 on p.\ 360 in \cite{MZ}), the set $E^{\text{\tiny{$M$}}}$ is a relatively compact subset of $\mathcal P (\Omega)$. 
Moreover, the projection map 
$$
p\colon E \to L\quad\text{with} \quad p(\nu^\rho):=\mathbb P \circ (X^\rho)^{-1}, \quad\text{for } \rho= ({\Omega}, {\mathcal{F}}, {\mathbb{F}}, {\mathbb{P}}, x_0 ,W, \xi) \in E^w, 
$$
satisfies the conditions in Assumption \ref{assumption.bestresponsemap.nonempty} and Assumption \ref{assumption.general.approximation}.

Let $2^L$ be the set of all subsets of $L$. Then, thanks to \eqref{sub mfg sing a priori estimate equation minimizers}, the best-response correspondence $R \colon L \to 2^L$, given by
$R( \mu) := \big\{ p \nu \, \big| \, \nu \in \argmin_E  J(\cdot,\mu) \big\}$ for $\mu \in L$, is well-defined.   
The flow of measures $\mu^* \in L$ is a solution to the mean field game with singular controls if $\mu^* \in R(\mu^*)$. 

\subsubsection{Existence and approximation of solutions} 
In order to employ the results from Section \ref{section general case}, we begin by providing the following technical result.   
\begin{lemma} 
\label{lemma singular control construction of operation}
There exists two operations $\land^{\text{\tiny{$E$}}}, \lor^{\text{\tiny{$E$}}}\colon E^{\text{\tiny{$M$}}} \times E^{\text{\tiny{$M$}}} \to E$ satisfying Assumption \ref{assumption}.
\end{lemma}

\begin{proof} 
The argument exploits an approximation scheme of the singular controls through regular controls and the results derived in \cite{dianetti.ferrari.fischer.nendel.2019}. 
We divide the proof in four steps.

\smallbreak\noindent  
\emph{Step 1.} 
For $i=1,2$, take control rules $\nu_i= \nu^{\rho_i} \in E^M$ with $\rho_i = (\Omega^i, \mathcal{F}^i, \mathbb F ^i,  \mathbb P ^i, x_0^i, W^i, \xi^i) \in E^{M,w}$.
Without loss of generality, we can assume that the controls $\rho_1, \rho_2$ are defined on a same stochastic basis $(\Omega, \mathcal{F}, \mathbb F ,  \mathbb P , x_0, W)$; that is, $(\Omega^i, \mathcal{F}^i, \mathbb F ^i,  \mathbb P ^i, x_0^i, W^i)=(\Omega, \mathcal{F}, \mathbb F ,  \mathbb P , x_0, W)$, for $i=1,2$. 

Introduce a Wong-Zakai-type approximation of $\xi^i$ by defining the sequences of processes $ ( \xi^{i,n} )_{n\in \mathbb{N}}$ through  
\begin{align}
\label{eq wong-zakai definition} 
     \xi_t^{i,n}:= 
     \begin{cases}
    n \int_{t-1/n}^t \xi_s^i ds,  & t \in [0,T), \\
     \xi_T^i, & t=T,
     \end{cases}
\end{align}
for each $n \in \mathbb{N}$.
Recall that processes are always (implicitly) assumed to be equal to $0$ for negative times. 
Further, note that, since $\mathbb E ^{\mathbb P} [|\xi_T^i|^p] < \infty$ (recall that $\xi^i \in E^{M,w}$ by assumption), the processes $\xi^{i,n}$ are Lipschitz continuous on the time interval $[0,T)$. However, they may have  discontinuities at time $T$. 
Moreover, for each $i=1,2$ and all $n \in \mathbb N$, denote by $X^{i,n}$ the solution to the controlled SDE
\begin{equation*}
X_t^{i,n}= x_0 + \int_0^t b(s,X_s^{i,n}) ds  + \sigma W_t +  \xi_t^{i,n}, \quad t \in [0,T].     
\end{equation*}

Next, since the processes $ \xi^{i,n}$ have Lipschitz paths and are nondecreasing, we can find $\mathbb F$-adapted processes $u^{i,n}\colon \Omega \times [0,T] \to [0,\infty)$ such that
$$
\xi_t^{i,n}= \int_0^t u_s^{i,n} ds, \quad t \in [0,T).
$$
Observing that the processes $u^{i,n}$ can be regarded as regular controls, we wish to employ the results from \cite{dianetti.ferrari.fischer.nendel.2019} in order to construct $\rho^\land, \, \rho^\lor$. 
However, we need to take care of possible discontinuities at time $T$.

As in Lemma 2.10 in \cite{dianetti.ferrari.fischer.nendel.2019}, for each $n \in \mathbb N$, we find two $\mathbb F$-adapted $[0,\infty)$-valued processes $u^{\land, n}, u^{\lor,n}$ such that, defining $\P$-a.s., 
\begin{equation}
\label{eq v inf n, v sup n}
 \xi_t^{\land,n}:=\int_0^t u_s^{\land,n} ds 
 \quad \text{and} \quad  
 \xi_t^{\lor,n}:=\int_0^t u_s^{\lor,n} ds, \quad \text{for each } t \in [0,T),
\end{equation}
we have, $\text{for each } t \in [0,T),\; \mathbb{P}\text{-a.s.}$, 
\begin{align}
\label{eq SDE for sup and inf n} 
    X_t^{1,n}\land X_t^{2,n} & = x_0 + \int_0^t b(s, X_s^{1,n} \land X_s^{2,n}) ds + \sigma W_t + \xi_t^{\land,n} \quad\text{and}  \\ \notag 
    X_t^{1,n} \lor X_t^{2,n} & = x_0 + \int_0^t b(s, X_s^{1,n} \lor X_s^{2,n}) ds + \sigma W_t + \xi_t^{\lor,n} .  
\end{align} 
This suggest to define the processes $\xi^{\land,n}$ and $\xi^{\lor,n}$ at time $T$ by setting, $\mathbb P$-a.s.,
\begin{align*}
    \xi_T^{\land,n} &:= X_T^{1,n}\land X_T^{2,n} - x_0 - \int_0^T b(s, X_s^{1,n} \land X_s^{2,n}) ds - \sigma W_T \quad\text{and}  \\ \notag 
    \xi_T^{\lor,n}  &:= X_T^{1,n} \lor X_T^{2,n} - x_0 - \int_0^T b(s, X_s^{1,n} \lor X_s^{2,n}) ds - \sigma W_T.  
\end{align*} 
We define
\begin{align*}
\rho^{\land, n} &:= ( \Omega, \mathcal F, \mathbb F, \mathbb P, x_0, W, \xi^{\land, n}), \\
 \rho^{\lor, n} &:= ( \Omega, \mathcal F, \mathbb F, \mathbb P, x_0, W, \xi^{\lor, n}),  \\
 \rho^{i, n}    &:= ( \Omega, \mathcal F, \mathbb F, \mathbb P, x_0, W, \xi^{i, n}), \quad\text{for } i=1,2,
\end{align*}
so that, by virtue of \eqref{eq SDE for sup and inf n} and the definition of $\xi_T^{\land,n}$ and $\xi_T^{\lor,n}$, we obtain, $\mathbb P$-a.s.,
\begin{equation}\label{eq mfgsing real SDE prelimit} 
   X_t^{1,n}\land X_t^{2,n} = X_t^{\rho^{\land,n}} \quad \text{and} \quad  X_t^{1,n} \lor X_t^{2,n} = X_t^{\rho^{\lor,n}}, \quad  \text{for any $t\in [0,T]$.}
\end{equation}
Moreover, we observe that the processes $\xi^{\land,n}$ and $\xi^{\lor,n}$ are nondecreasing. 

\smallbreak\noindent
\emph{Step 2}. In this step, we prove that
\begin{equation}\label{eq J submod befor limits}
J(\rho^{\land,n}, \mu) +J(\rho^{\lor,n}, \mu) = J(\rho^{1,n}, \mu) + J(\rho^{2,n}, \mu).
\end{equation}
This is again done by adapting arguments from \cite{dianetti.ferrari.fischer.nendel.2019}, taking care of possible discontinuities of the processes $\xi^{i,n}, \ \xi^{\land,n}, \ \xi^{\lor,n} $ at time $T$. 

For a generic admissible control $\rho= (\Omega, \mathcal{F}, \mathbb F, \mathbb P, x_0, W, \xi) \in E^w$, using integration by parts and the controlled SDE \eqref{dynamics.singular.controls}, we rewrite the cost functional as
\begin{align}\label{eq mfgsing Jrepresent}
J(\xi,\mu) & = \mathbb E ^{\mathbb P} \bigg[ \int_0^T f(t,X_t^{\rho}, \mu_t) dt  + g(X_T^{\rho}, \mu_T) + c_T \xi_T - \int_0^T \xi_t c_t' dt \bigg] \\ \notag
& = \mathbb E ^{\mathbb P} \bigg[ \int_0^T \big( f(t,X_t^{\rho}, \mu_t) - c_T b(t,X_t^{\rho} ) - \xi_t c_t'\big) dt  + g(X_T^{\rho}, \mu_T) + c_T X_T^{\rho} \bigg]  - c_T \mathbb E ^{\mathbb P} [x_0] \\ \notag
& = G^1 ( \rho, \mu ) - G^2 (\rho, \mu)  + H(\rho,\mu) - c_T \mathbb E ^{\mathbb P} [x_0], \notag
\end{align}
where we have set
\begin{align*}
    G^1 ( \rho, \mu ) & := \mathbb E ^{\mathbb P} \bigg[ \int_0^T \big( f(t,X_t^{\rho}, \mu_t) - c_T b(t,X_t^{\rho} )\big) dt \bigg], \\
    G^2 (\rho, \mu) & := \mathbb E ^{\mathbb P} \bigg[ \int_0^T \xi_t c_t' dt \bigg], \\
     H(\rho,\mu) & := \mathbb E ^{\mathbb P} [  g(X_T^{\rho}, \mu_T) + c_T X_T^{\rho} ].
\end{align*}
Observing that the functional $G^1$ depends on the control only on the interval $[0,T)$, thanks to the construction of $u^{\land,n},\, u^{\lor,n} $ provided in the Step 1, we can repeat the arguments in the proof of Lemma 2.11 in \cite{dianetti.ferrari.fischer.nendel.2019} in order to come up with
\begin{equation}\label{eq J submod befor limits 1}
G^1(\rho^{\land,n}, \mu) + G^1(\rho^{\lor,n}, \mu) = G^1(\rho^{1,n}, \mu) + G^1 (\rho^{2,n}, \mu).
\end{equation}
Moreover, from the definition of $u^{\land, n}$ and $u^{\lor, n}$ in Step 1, as in the proof of  Lemma 2.11 in \cite{dianetti.ferrari.fischer.nendel.2019}, we see that
$$
\xi_t^{\land, n} + \xi_t^{\lor, n} = \int_0^t (u_s^{\land, n} + u_s^{\lor, n}) ds = \int_0^t (u_s^{1, n} + u_s^{2, n}) ds =\xi_t^{1, n} + \xi_t^{2, n}, \quad \text{for each $t \in [0,T)$},
$$
so that
\begin{equation}\label{eq J submod befor limits 2}
G^2(\rho^{\land,n}, \mu) + G^2(\rho^{\lor,n}, \mu) = G^2(\rho^{1,n}, \mu) + G^2 (\rho^{2,n}, \mu).
\end{equation}
Finally, we easily find that
\begin{equation}\label{eq J submod befor limits 3}
H(X_T^{\rho^{\land,n}}, \mu) + H(X_T^{\rho^{\lor,n}}, \mu) = H (X_T^{\rho^{1,n}}, \mu) + H (X_T^{\rho^{2,n}},\mu).
\end{equation}
Therefore, adding \eqref{eq J submod befor limits 1}, \eqref{eq J submod befor limits 2}, and \eqref{eq J submod befor limits 3}, and using the representation in \eqref{eq mfgsing Jrepresent}, we obtain \eqref{eq J submod befor limits}.

\smallbreak\noindent
\emph{Step 3.} 
Set $X^i:=X^{\rho_i}$, $i=1,2$, and define the right-continuous processes $\xi^\land, \ \xi^\lor$ by setting 
\begin{align}
    \xi_t^\land & := X_t^1\land X_t^2 - x_0 - \int_0^t b(s, X_s^1 \land X_s^2) ds - \sigma W_t, \\ \notag
    \xi_t^\lor & := X_t^1\lor X_t^2 - x_0 - \int_0^t b(s, X_s^1 \lor X_s^2) ds - \sigma W_t .
\end{align}
The aim of this step is to prove that the controls $\rho^\land:=(\Omega, \mathcal{F}, \mathbb F ,  \mathbb P , x_0, W, \xi^\land)$ and $\rho^\lor:=(\Omega, \mathcal{F}, \mathbb F ,  \mathbb P , x_0, W, \xi^\lor)$ are admissible, and that the control rules 
$$ 
\nu_1 \land^{\text{\tiny{$E$}}} \nu_2 :=\nu^{\rho^\land} \quad \text{and} \quad  \nu_1 \lor^{\text{\tiny{$E$}}} \nu_2 :=\nu^{\rho^\lor}
$$
satisfy the conditions in Assumption \ref{assumption}.

From \eqref{eq wong-zakai definition}, we immediately see that, $\mathbb P$-a.s.,
\begin{equation}\label{eq subsing xi n to xi}
\begin{cases}
\xi_t^{i,n} \to \xi_t^i \text{ as $n\to \infty$ for all continuity points $t \in [0,T)$ of $\xi^i$},    \\
\xi^{i,n}_T \to \xi^i_T  \text{ as $n\to \infty$}.
\end{cases} 
\end{equation}
Therefore, using \eqref{eq subsing xi n to xi} and Gr\"onwall's inequality, we deduce that,  $\mathbb P$-a.s.,
\begin{equation}\label{eq subsing X n to X}
\begin{cases}
X_t^{i,n} \to X_t^i \text{ as $n\to \infty$ for all continuity points $t \in [0,T)$ of $X^i$},    \\
X^{i,n}_T \to X^i_T  \text{ as $n\to \infty$}.
\end{cases} 
\end{equation} 
This allows to take limits in \eqref{eq mfgsing real SDE prelimit} in order to conclude that, $\mathbb P$-a.s., for $\pi$-a.a.\ $t \in [0,T]$, we have
\begin{equation}\label{eq X1 land X2 to X1 land X2}
    X_t^{1,n}\land X_t^{2,n} \to X_t^{1}\land X_t^{2}, \quad  X_t^{1,n}\lor X_t^{2,n} \to X_t^{1}\lor X_t^{2}, \quad \xi_t^{\land,n} \to \xi_t^\land,  \quad \xi_t^{\lor,n} \to \xi_t^\lor. 
\end{equation} 
Since the processes $\xi ^{\land,n}$ and $\xi^{\lor,n}$ are nonnegative and nondecreasing, also the limit processes $\xi^\land$ and $\xi^\lor$ are nonnegative and nondecreasing, hence $\rho^\land$ and $\rho^\lor$ are admissible. 
Moreover, by definition of $\xi^\land$ and $\xi^\lor$, we have 
\begin{equation*}
 X_t^{\rho^\land} = X_t^1 \land X_t^2 \leq   X_t^1 \lor X_t^2 \leq X_t^{\rho^\lor},\; \mathbb P \text{-a.s.}, \quad \text{for each $t \in [0,T]$,} 
\end{equation*}
which proves that $\nu_1 \land^{\text{\tiny{$E$}}} \nu_2$ and $\nu_1 \lor^{\text{\tiny{$E$}}} \nu_2$ satisfy Condition \ref{assumption.existence.operation} in Assumption \ref{assumption}. 

\smallbreak\noindent
\emph{Step 4.} We conclude by proving that $\nu_1 \land^{\text{\tiny{$E$}}} \nu_2$ and $\nu_1 \lor^{\text{\tiny{$E$}}} \nu_2$  satisfy Condition \ref{assumption.submod.J} in Assumption \ref{assumption}.  
We begin by observing that, for a generic constant $C>0$, by Gr\"onwall's inequality, we have
$$
|X_t^{i,n}|^p \leq C \Big ( 1+ |x_0|^p + \sigma^p \sup_{s \in [0,T]} |W_s|^p  +    |\xi_T^{i,n}|^p \Big),
$$
so that, by definition of $\xi^{i,n}$, we obtain
\begin{equation}\label{eq subsing dominate estimate}
\sup_n \sup_{t \in [0,T]} |X_t^{i,n}|^p \leq C \Big ( 1+ |x_0|^p + \sigma^p \sup_{s \in [0,T]} |W_s|^p  +    |\xi_T^{i}|^p \Big) \in \mathbb L^1 (\Omega ; \mathbb P), 
\end{equation}
where the integrability condition of the right hand side follows from the fact that  $\nu_1, \, \nu_2 \in E^{\text{\tiny{$M$}}}$. 
Therefore, thanks to the convergences in \eqref{eq subsing xi n to xi} and \eqref{eq subsing X n to X} and the estimate \eqref{eq subsing dominate estimate}, the growth conditions on $f$ and $g$ allows to employ the dominated convergence theorem in order to come up with
\begin{align}\label{eq singular lim J submodular 0}
J(\rho_i, \mu)   
& = \mathbb E ^{\mathbb P} \bigg[ \int_0^T f(t,X_t^{\rho_i}, \mu_t) dt  + g(X_T^{\rho_i}, \mu_T) + \int_{[0,T]} c_t d\xi^{i}_t \bigg] \\ \notag 
& = \lim_n \mathbb E ^{\mathbb P} \bigg[ \int_0^T f(t,X_t^{\rho^{i,n}}, \mu_t) dt  + g(X_T^{\rho^{i,n}}, \mu_T) + \int_{[0,T]} c_t d\xi^{i,n}_t \bigg]  = \lim_n J(\rho^{i,n}, \mu).
\end{align}
Now, an integration by parts together with the limit behaviour in \eqref{eq X1 land X2 to X1 land X2} and Fatou's lemma yields the estimate
\begin{align}\label{eq singular lim J submodular 1}
J(\rho^\land, \mu) 
& = \mathbb E ^{\mathbb P} \bigg[ \int_0^T f(t,X_t^{\rho^\land}, \mu_t) dt  + g(X_T^{\rho^\land}, \mu_T) + c_T \xi_T^\land - \int_0^T \xi^\land_t c_t' dt  \bigg] \\ \notag
& \leq \liminf_n \mathbb E ^{\mathbb P} \bigg[ \int_0^T f(t,X_t^{\rho^{\land,n}}, \mu_t) dt  + g(X_T^{\rho^{\land,n}}, \mu_T) + c_T \xi_T^{\land,n} - \int_0^T \xi^{\land,n}_t c_t' dt  \bigg]\\ \notag 
& = \liminf_n J(\rho^{\land,n}, \mu).
\end{align}     
Similarly, it follows that
\begin{align}\label{eq singular lim J submodular 2} 
J(\rho^\lor, \mu) \leq \liminf_n J(\rho^{\lor,n}, \mu). 
\end{align}  
Finally, exploiting \eqref{eq singular lim J submodular 0}, \eqref{eq singular lim J submodular 1}, and \eqref{eq singular lim J submodular 2}, we can take limits in \eqref{eq J submod befor limits} in order to obtain Condition \ref{assumption.submod.J} in Assumption \ref{assumption}.
\end{proof}  

Thanks to Lemma \ref{lemma singular control construction of operation} and Remark \ref{remark singular control BRM nonempty}, we see that all Assumption \ref{assumption.bestresponsemap.nonempty} and \ref{assumption} are satisfied. As a consequence of Theorem \ref{theorem.main.general} we have the following result.
\begin{theorem}\label{thm singular weak formulation} 
The set of mean field game equilibria $\mathcal M$ is nonempty with $\inf \mathcal M \in \mathcal M$ and $\sup \mathcal M \in \mathcal M$. Moreover, if $f$ and $g$ are continuous in $(x,\mu)$, then
\begin{enumerate}
 \item the learning procedure $\underline{\mu}^n$ defined inductively by $\underline{\mu}^0=\inf L$ and $\underline{\mu}^{n+1}=\inf R(\underline{\mu}^n)$ is nondecreasing in $L$ and it converges to the minimum MFG solution,
 \item the learning procedure $\overline{\mu}^n$ defined inductively by $\overline{\mu}^0=\sup L$ and $\overline{\mu}^{n+1}=\sup R(\overline{\mu}^n)$ is nonincreasing in $L$ and it converges to the maximal MFG solution.  
\end{enumerate}
\end{theorem}
\subsection{Remarks and extensions} 
The previous arguments can be easily adapted in order to cover many classical settings, which typically arise in the literature on stochastic singular control, such as, for example, MFGs where the optimization problem concerns an infinite time-horizon discounted criterion or involves controls of bounded variation, rather than just monotone. A similar setting has been, for example, considered in \cite{Guo&Xu18}. 
In the following, we illustrate a few specific settings of interest.  

\begin{remark}[Controlled Ornstein-Uhlenbeck process and common noise]
We underline that the results of Subsection \ref{section geometric BM} can also be obtained if the underlying dynamics is given by a controlled Ornstein-Uhlenbeck process; that is, if the state process evolves according to
\begin{equation*}
dX_t^\xi = \theta (  \lambda - X_t^\xi ) dt + \sigma dW_t  + \sigma^o dB_t + d\xi_t, \ t\in [0,T], \quad X_{0-}^\xi = x_0, \ 
\end{equation*} 
with $\kappa,\, \lambda \in \mathbb R, \ \sigma, \, \sigma^o \geq 0$. In this case, the state process can be explicitly written as
$$
X_t^\xi= e^{-\theta t}\bigg( x +  \lambda (e^{\theta t} -1) + \int_0^t e^{\theta s } (\sigma dW_s + \sigma^o dB_s) + \int_{[0,t]} e^{\theta s} d\xi_s \bigg),
$$
and, for $\xi,\, \bar{\xi} \in \mathcal{ V}_{\uparrow}$, we have $X^\xi \land X^{\bar{\xi}} = X^{\xi^\land} $ and $X^\xi \lor X^{\bar{\xi}} = X^{\xi^\lor} $ by setting
$$
\xi_t^\land:= \int_{[0,t]} e^{-\theta s} d (\zeta \land \bar{\zeta})_s, \quad 
\xi_t^\lor:= \int_{[0,t]} e^{-\theta s} d (\zeta \lor \bar{\zeta})_s, \quad
\zeta:= \int_{[0,t]} e^{\theta s} d\xi_s, \quad 
\bar \zeta:= \int_{[0,t]} e^{\theta s} d \bar \xi_s.
$$
Therefore one can introduce, as in \eqref{equation singular definition operation linear}, operations that satisfy all the requirements from Assumption \ref{assumption}.
\end{remark}

\begin{remark}[Mean-field-dependent dynamics and relation to \cite{Campietal}] The approach from Subsection \ref{section geometric BM} also allows to cover problems, where the drift of the underlying state process depends in an increasing way (w.r.t.\ first-order stochastic dominance) on the mean field, in such a way that \ref{tightness.singular.geometric} holds true. This could be, for example, achieved if $b$ in \eqref{SDE.singular.geometric} is replaced by a bounded increasing function of $(\mu_t)_{t\in[0,T]}$.

Another example is given by the two-dimensional MFG of finite-fuel capacity expansion considered in \cite{Campietal}. Therein, the mean of a uniformly bounded purely controlled process affects in a nondecreasing way the drift of an uncontrolled It\^o-diffusion and there is no mean field dependence in the profit functional. We refer to Remark 3.15 in \cite{Campietal} for additional details on how the existence of a mean field equilibrium for the problem considered in that paper can be indeed achieved via our lattice-theoretic techniques.
\end{remark}

\section{Submodular mean field games with reflecting boundary conditions}
\label{section reflected diff}

In this section, we consider a MFG model with reflecting boundary conditions, in which the state process of the representative player is forced to remain in a certain interval of the state space.
These types of models were recently introduced in \cite{bayraktar.budhiraja.cohen.2019.AAP} (see also \cite{bayraktar.budhiraja.cohen.2018numerical}), motivated by applications to queueing systems consisting of many strategic servers that are weakly interacting.
Also, a particular setting in the same class of models is studied in \cite{graber.mouzouni.2020}, motivated by a model for the production of exhaustible resources.
Here, we consider a version of the model in \cite{bayraktar.budhiraja.cohen.2019.AAP} with submodular cost, which we solve through the results of Section \ref{section general case}. 

\subsection{Formulation of the model}  
Fix $M>0$ and $x_0 \in [0,M]$. 
Consider the set $L_{\text{\tiny{$M$}}}$ of all measurable functions $\mu:[0,T] \to \text{$\mathcal{P}([0,M])$ with $\mu_0 = \delta_{x_0}$}$,
endowed with the lattice structure coming from the order relation $\leq^{\text{\tiny{$L$}}}$ of $\pi := \delta _0 + dt + \delta_T$-pointwise first order stochastic dominance. As in the previous section, this leads to a complete lattice $(L_{\text{\tiny{$M$}}}, \leq^{\text{\tiny{$L$}}})$. 

Next, we introduce the minimization problem. 
For technical reasons (i.e., in order to gain compactness of the set of controls), we do so by using relaxed controls, though we work with assumptions under which strict optimal controls always exist. 
For a compact control set $A \subset \R$, and  a Lipschitz continuous function $b\colon [0,T] \times \R \to \R$, we define the set of \emph{admissible relaxed controls} as the set $E^{w}$ of  tuples $\rho:=(\Omega,\mathcal F, \mathbb F, \mathbb P, W , \lambda, v, X)$ such that  
\begin{enumerate}
    \item $W=(W_t)_{t\in [0,T]}$ is a Brownian motion on the filtered probability space $(\Omega,\mathcal F, \mathbb F, \mathbb P)$, satisfying the usual conditions,
    \item $\lambda$ is a $\mathcal P (A)$-valued, progressively measurable process,
    \item the couple $(v,X)$ is a solution to the controlled reflected SDE in the domain $(0,M)$:
    \begin{equation}\label{eq SDE reflected controlled} 
    \begin{cases}  
    dX_t = \big( b(t,X_t)  + \begin{matrix} \int_A a \lambda_t (da) \end{matrix} \big) dt  + \sigma d W_t + dv_t, \ t \in [0,T], \quad X_0 = x_0, \\ 
    X_t \in [0,M], \quad 
    \begin{matrix} \int_0^t \mathds 1 _{\{ X_s \in (0,M) \}} d |v|_s \end{matrix}=0,  \ \text{for any } t\in [0,T],\ \mathbb{P}\text{-a.s.}, \\
    \end{cases} 
    \end{equation} 
\end{enumerate}
where $|v|$ denotes the total variation of $v$.
Moreover, we define the set of \emph{admissible strict controls} $E^{w,s}$ as the set of elements $\rho:=(\Omega,\mathcal F, \mathbb F, \mathbb P, W, \lambda, v,X) \in E^{w}$ such that $\lambda_t = \delta_{\alpha_t}$ $\mathbb P \otimes dt$-a.e.\ in $\Omega \times [0,T]$, for some $A$-valued progressively measurable process $\alpha$.
 
We consider functions $f, \, g$, and $c$ as in the beginning of Section \ref{section singular controls} satisfying Assumption \ref{ass.submodularity.continuos.time setting}, and a lower semicontinuous function
$l \colon [0,T] \times \mathbb{R} \times \R \to  [0,\infty)$, which is convex in $a$.
For $\mu \in L_{\text{\tiny{$M$}}}$ and $\rho=(\Omega,\mathcal F, \mathbb F, \mathbb P, W, \lambda, v,X) \in E^{w}$, we define the cost functional
\begin{equation*}
J(\rho,\mu):= \mathbb E ^{\mathbb P} \bigg[ \int_0^T \Big( f(t,X_t,\mu_t) + \int_A l(t, X_t,a) \lambda_t(da) \Big) dt + g(X_T,\mu_T) + \int_0^T c_t d|v|_t \bigg].
\end{equation*}
We say that $\rho \in E^{w} $ is an \emph{optimal singular control} for the flow of measures $\mu$ if $ {J}(\rho,\mu) =\inf_{E^{w}} {J}(\cdot,\mu)$.

We are interested in the following notion of equilibrium.
\begin{definition}\label{definition reflected mfg eq strict} A flow of probabilities ${\mu} \in L_{\text{\tiny{$M$}}}$ is a MFG equilibrium if
\begin{enumerate} 
    \item there exists a strict optimal control $\rho=(\Omega,\mathcal F, \mathbb F, \mathbb P, W, \lambda, v,X) \in E^{w,s} $  for $\mu$,
    \item ${\mu}_t= \mathbb{P} \circ (X_t)^{-1}$ for any $t\in [0,T]$.  
\end{enumerate}
\end{definition} 

\subsection{Reformulation via control rules and preliminary results}
In order to have a topology on the space of admissible controls, we reformulate the problem in terms of control rules. 

Introduce the canonical space $(\Omega, \mathcal{F})$, where
\begin{align*}
\Omega := \mathcal C  \times \Lambda \times \mathcal{V} \times \mathcal{D}, \quad
\mathcal{F} := \mathcal{B}(\mathcal{C})  \otimes \mathcal{B}({\Lambda}) \otimes \mathcal{B}(\mathcal{V}) \otimes \mathcal{B}(\mathcal{D}).
\end{align*}
Define the set of relaxed control rules 
$$
E:= \{ \nu^\rho\, |\,  \rho \in E^{w} \} \quad \text{with}\quad \nu^\rho:= \mathbb P \circ (W,\lambda, v, X)^{-1}, \text{ for } \rho = ({\Omega}, {\mathcal{F}}, {\mathbb{F}}, {\mathbb{P}}, W,\lambda, v,  X) \in E^{w},
$$
and, with a slight abuse of notation, we set $J(\nu^\rho,\mu):= J(\rho, \mu)$.
The set of strict control rules is defined as $E^s:= \{ \nu^\rho\, |\,  \rho \in E^{w,s} \}$.
In this way, $E$ is naturally defined as a subspace of the topological space $\mathbb P (\Omega)$.
For any $\rho = ({\Omega}, {\mathcal{F}}, {\mathbb{F}}, {\mathbb{P}}, W,\lambda, v,  X) \in E^{w}$, the controlled SDE \eqref{eq SDE reflected controlled} together with $X_t \in [0,M]$ implies the estimate 
\begin{equation}\label{eq reflected a priori estimate}
\mathbb E ^{\mathbb P}[ |v|_T^p] \leq K < \infty
\end{equation} 
with a constant $K>0$. 
Moreover, since $A$ is compact, so are $\Lambda$ and $\mathcal P (\Lambda)$. 
This, together with \eqref{eq reflected a priori estimate}, allows to use the Meyer-Zheng tightness criteria (see Theorem 4 on p.\ 360 in \cite{MZ}) to show that the set $E$ is a relatively compact subset of $\mathcal P (\Omega)$.  
Moreover, the projection map  
$$
p\colon E \to L_{\text{\tiny{$M$}}}\quad \text{with} \quad p(\nu^\rho):=\mathbb P \circ (X^\rho)^{-1}, \text{ for } \rho= ({\Omega}, {\mathcal{F}}, {\mathbb{F}}, {\mathbb{P}}, W, \lambda, v, X) \in E^w,   
$$
satisfies the conditions in Assumption \ref{assumption.bestresponsemap.nonempty} and Assumption \ref{assumption.general.approximation}.

\begin{lemma}
\label{lemma reflected existence optimal controls} \ 
\begin{enumerate}
    \item\label{lemma reflected existence optimal controls 1} For any ${\mu} \in L_{\text{\tiny{$M$}}}$, the set $ \argmin_{E} J(\cdot, \mu)$ is nonempty. 
    \item\label{lemma reflected existence optimal controls 3} If $\rho = ({\Omega}, {\mathcal{F}}, {\mathbb{F}}, {\mathbb{P}}, W,  X,\lambda, v) \in E^{w}$, there exists a control $\hat{\rho} = ({\Omega}, {\mathcal{F}}, {\mathbb{F}}, {\mathbb{P}}, W,  X, \hat \lambda, v) \in E^{w,s}$ such that $J(\hat{\rho}, \mu) \leq J(\rho, \mu)$, for any ${\mu} \in L_{\text{\tiny{$M$}}}$. 
\end{enumerate} 
\end{lemma} 
\begin{proof} We begin by proving Claim \ref{lemma reflected existence optimal controls 1}.
In order to do so, take a minimizing sequence $(\nu_n)_n \subset E$ (i.e., $\lim_n J(\nu_n,\mu) = \inf_{E} J(\cdot, \mu)$)
and controls $\rho_n = (\Omega^n, \mathcal{F}^n, \mathbb F ^n,  \mathbb P ^n, W^n, \lambda^n, v^n, X^n) \in E ^{w}$ with $\nu_n = \nu^{\rho_n}$.
Since the set $E \subset \mathcal P (\Omega)$ is relatively compact, we can find a limit point 
$\nu_* \in \mathcal P ( \mathcal C \times \Lambda \times \mathcal V  \times \mathcal D) $ and a subsequence (not relabelled) such that $\nu_n \to \nu_*$ weakly.  
Up to using a Skorokhod representation theorem for separable spaces (see Theorem 3 in \cite{d}), we can assume that there exists a common probability space $(\Omega, \mathcal{F}, \mathbb P)$, on which the processes $(W^n,\lambda^n, v^n, X^n)$ are defined together with a process $(W, \lambda, v , X)$, such that 
\begin{align}\label{eq reflected Skorokhod convergence} 
& (W^n , \lambda^n, v^n, X^n) \to (W, X,\lambda, v),\; \P\text{-a.s., in $\mathcal C  \times \Lambda \times \mathcal V \times \mathcal D$ as $n\to \infty$,} \\ \notag
& \text{$\nu_n = \mathbb P \circ (W^n,\lambda^n, v^n, X^n)^{-1}$ and $\mathbb P \circ (W, \lambda, v, X)^{-1} = \nu_*$.} 
\end{align}
Also, this convergence allows  to show that $X$ is a solution to the SDE $X_t = x_0 + \int_0^t \big( b(s,X_s) + \int_A  a \lambda_s(da) \big) ds + \sigma W_t + v_t$, $t \in [0,T]$, $\mathbb P$-a.s.
Moreover, by the Lipschitz continuity of the Skorokhod map (see Lemma 2.1 in \cite{bayraktar.budhiraja.cohen.2019.AAP}), we see that the couple $(v,X)$ solves the controlled reflected SDE \eqref{eq SDE reflected controlled}.
Therefore, defining $\rho_* = (\Omega, \mathcal{F}, \mathbb F ,  \mathbb P ,  W, \lambda, v, X)$ with $\mathbb F$ being the (extended) filtration generated by $(W, \lambda, v,X)$, we have that $\rho_* \in E^{w}$ and $\nu_*:= \nu^{\rho_*}$.  
Moreover, using the convergence in \eqref{eq reflected Skorokhod convergence} and exploiting the lower semicontinuity of the costs $f,g,l$ and the fact that $c$ is nondecreasing, by Fatou's lemma we obtain that
$$
J(\nu_*, \mu) \leq \liminf_n J(\nu_n,\mu) =\inf_{E} J(\cdot, \mu),
$$
which completes the proof of Claim \ref{lemma reflected existence optimal controls 1}.

We conclude by proving Claim \ref{lemma reflected existence optimal controls 3}. 
Take $\rho = (\Omega, \mathcal{F}, \mathbb F ,  \mathbb P , W, X, \lambda, v)\in E^{w}$, set $\alpha_t := \int_A a \lambda_t(da)$, $\hat{\lambda} := \delta_{\alpha_t} (da) dt$, and consider the control $\hat \rho = (\Omega, \mathcal{F}, \mathbb F ,  \mathbb P , W, X, \hat \lambda, v) \in E^{w,s}$.  
First of all, we see that, $\P$-a.s., $X$ solves the equation 
$$
X_t = x_0 + \int_0^t ( b(s,X_s) + \alpha_s) ds + \sigma W_t + v_t, \quad t \in [0,T].
$$
Finally, by convexity of $l$ we can use Jensen's inequality obtaining
\begin{align*}
J(\hat \rho, \mu) &=\mathbb E ^{\mathbb P} \bigg[ \int_0^T ( f(t,X_t,\mu_t) +  l (t, X_t,\alpha_t) ) dt + g(X_T,\mu_T) + \int_0^T c_t d|v|_t \bigg] \\
& \leq \mathbb E ^{\mathbb P} \bigg[ \int_0^T \bigg( f(t,X_t,\mu_t) + \int_A l (t, X_t,a) \lambda_t (da) \bigg) dt + g(X_T,\mu_T) + \int_0^T c_t d|v|_t \bigg]  \\
& = J(\rho, \mu),
\end{align*}
which completes the proof of the lemma.
\end{proof} 

\subsection{Existence and approximation of equilibria}
We begin by observing that a \emph{relaxed MFG equilibrium} can now be seen as a fixed point of the best-response-map 
\begin{equation}\label{eq reflected MFG equilibrium relaxed}
 R\colon L_{\text{\tiny{$M$}}} \to L_{\text{\tiny{$M$}}}\quad\text{with} \quad R(\mu):=p(\begin{matrix} \argmin_E \end{matrix} J(\cdot, \mu)), \text{ for } \mu \in L_{\text{\tiny{$M$}}}.
\end{equation}

We move on by constructing operations $\land^{\text{\tiny{$E$}}}, \lor^{\text{\tiny{$E$}}} \colon E \times E \to E$ satisfying Assumption \ref{assumption}.  
For $\nu= \nu^\rho, \bar \nu= \nu^{\bar \rho} \in E$ with $\rho = (\Omega,\mathcal F, \mathbb F, \mathbb P, W,\lambda ,v,X), \, \bar{\rho}=(\bar \Omega,\bar{ \mathcal F}, \bar {\mathbb F}, \bar {\mathbb P}, \bar W,\bar{\lambda},\bar{v},\bar{X}) \in E^w$, 
we can, without loss of generality (see, e.g., the proof of Lemma 3.4 in \cite{dianetti.ferrari.fischer.nendel.2019}), assume these controls to be defined on the same stochastic basis; that is,  $(\Omega,\mathcal F, \mathbb F, \mathbb P, W)=(\bar \Omega,\bar{ \mathcal F}, \bar {\mathbb F}, \bar {\mathbb P}, \bar W)$.
Hence, define 
$$
\rho^\land :=(\Omega,\mathcal F, \mathbb F, \mathbb P,W,\alpha^ \land ,v^ \land , X \land \bar{X}) 
\quad \text{and} \quad 
\rho^\lor :=(\Omega,\mathcal F, \mathbb F, \mathbb P,W, \alpha^ \lor, v^ \lor , X \lor \bar{X}),
$$ 
where 
\begin{equation*}    
\begin{matrix} 
& \lambda_t^\lambda := \lambda_t \mathds{1}_{ \{ X_t \leq \bar{X}_t \}} + \bar{\lambda}_t \mathds{1}_{ \{ X_t > \bar{X}_t \}}, 
\quad      
& v_t^\land := \int_0^t \big( \mathds{1}_{ \{X_s < \bar{X}_s \} } dv_s + \mathds{1}_{ \{X_s \geq \bar{X}_s \} }d \bar{v}_s \big),\\ 
& \lambda_t^\lor := \bar{\lambda}_t \mathds{1}_{ \{ X_t \leq \bar{X}_t \}} + {\lambda}_t \mathds{1}_{ \{ X_t > \bar{X}_t \}},  \quad
& v_t^\lor := \int_0^t \big( \mathds{1}_{ \{X_s < \bar{X}_s \} } d\bar{v}_s + \mathds{1}_{ \{X_s \geq \bar{X}_s \} }d{v}_s \big). 
\end{matrix}    
\end{equation*} 
Indeed, by the Meyer-It\^o formula for continuous semimartingales (see, e.g., Theorem 68 on p.\ 213 in \cite{Protter05}), we find  
\begin{align*}
X_t \land \bar{X}_t &= X_t + 0 \land ( \bar{X}_t - X_t ) \\
& = X_t + \int_0^t  \mathds{1}_{ \{\bar{X}_s -X_s \leq 0 \}} d(\bar{X} -X)_t  - \frac{1}{2}L_t^0(\bar{X} -X)\\ 
& = x_0 + \sigma W_t  \\
&\quad  +  \int_0^t \Big( \mathds{1}_{ \{X_s <\bar{X}_s \} } \Big( b(s,X_s) + \int_A a \lambda_s(da)\Big)  + \mathds{1}_{ \{X_s \geq \bar{X}_s \} } \Big( b(s,\bar{X}_s) + \int_A a  \bar \lambda _s(da)\Big) \Big) ds  \\ 
& \quad + \int_0^t \Big( \mathds{1}_{ \{X_s < \bar{X}_s \} } dv_s + \mathds{1}_{ \{X_s \geq \bar{X}_s \} }d \bar{v}_s \Big) 
- \frac{1}{2}L_t^0(\bar{X} -X),
\end{align*}
where $L_t^0(\bar{X} -X)$ is the local time of $\bar{X}-X$ at 0 (see, e.g., Chapter IV in \cite{Protter05}).
We denote by $[\bar{X} - X,\bar{X} - X ]$ the quadratic variation of the process $\bar X  -X$ (see, e.g., p. 66 in \cite{Protter05}). 
Since $\bar{X} - X$ is a process of bounded variation, we have $[\bar{X} - X,\bar{X} - X ]=0$.
Therefore, using the characterization of local times (see, e.g., Corollary 3 on p.\ 225 in \cite{Protter05}), we obtain that
$$
L_t^0 = \lim_{\varepsilon \to 0} \frac{1}{\varepsilon} \int_0^t \mathds{1}_{\{ 0 \leq \bar{X}_s - X_s \leq \varepsilon \}} d[\bar{X} - X,\bar{X} - X ]_s = 0,
$$
and conclude that
$$
X_t \land \bar{X}_t = x_0 + \int_0^t \Big( b(s,X_s\land \bar{X}_s) + \int_A a \lambda_s^\land(da)\Big)  ds + \sigma W_t  + v_t^\land.
$$
In the same way, the process $X \lor \bar{X}$ solves the SDE controlled by $\lambda ^\lor$ with reflection $v^\lor$.  Finally, $X_t \land \bar{X}_t, \, X_t \lor \bar{X}_t \in [0,M]$, and it can be easily verified that the support of the random measures $|v^\land|$ or \ $|v^\lor|$ is contained in the set of times at which $ X_t \land \bar{X}_t \in \{0,M\}$ or $ X_t \lor \bar{X}_t \in \{ 0,M \}$, respectively. 
This proves that $\rho^ \land , \, \rho^ \lor  \in E$, so that, defining $\nu \land^{\text{\tiny{$E$}}} \bar{\nu}:= \nu^{\rho^\land}, \, \nu \lor^{\text{\tiny{$E$}}} \bar \nu := \nu^{\rho^\lor}$, we have $\nu \land^{\text{\tiny{$E$}}} \bar{\nu}, \, \nu \lor^{\text{\tiny{$E$}}} \bar \nu \in E$.     

Moreover, one readily verifies that $|v^\land|_t + |v^\lor|_t \leq |v|_t + | \bar{v}|_t$. 
This together with the fact that $c'\leq 0$, in turn, yields the estimate
\begin{align*}
J(\nu \lor^{\text{\tiny{$E$}}} \bar{\nu} , \bar{\mu}) - J(\bar{\nu} , \bar{\mu}) \leq J(\nu \lor^{\text{\tiny{$E$}}} \bar{\nu} , {\mu}) - J( \bar{\nu}, {\mu}) \leq J(\nu , {\mu}) - J( \nu \land^{\text{\tiny{$E$}}} \bar{\nu} , {\mu}).  
\end{align*}
Hence, Assumption \ref{assumption} is satisfied. 

We can now state the main result of this section.
\begin{theorem}\label{theorem reflecting} The set of mean field game equilibria $\mathcal M$ is a non empty with $\inf \mathcal M \in \mathcal M$ and $\sup \mathcal M \in \mathcal M$. 
 Moreover, If $f$ and $g$ are continuous in $(x,\mu)$, then
\begin{enumerate}
 \item the learning procedure $\underline{\mu}^n$ defined inductively by $\underline{\mu}^0=\inf L$ and $\underline{\mu}^{n+1}=\inf R(\underline{\mu}^n)$ is nondecreasing in $L$ and it converges to the minimum MFG solution,
 \item the learning procedure $\overline{\mu}^n$ defined inductively by $\overline{\mu}^0=\sup L$ and $\overline{\mu}^{n+1}=\sup R(\overline{\mu}^n)$ is nonincreasing in $L$ and it converges to the maximal MFG solution.  
 \end{enumerate} 
\end{theorem} 
\begin{proof}
For relaxed MFG equilibria as in \eqref{eq reflected MFG equilibrium relaxed}, the result follows from the general Theorem \ref{theorem.main.general}. Thanks to Lemma \ref{lemma reflected existence optimal controls}, this allows to obtain the result for MFG equilibria as in Definition \ref{definition reflected mfg eq strict}. 
\end{proof}


\section{Supermodular mean field games with optimal stopping}
\label{section OS}

In this section we adapt the general results of Section \ref{section general case} to a MFG, where the representative agent faces an optimal stopping maximization problem. 
In particular, we introduce and solve a version of the model discussed in \cite{bouveret.dumitrescu.tankov.20} to which we add a common noise (see Example \ref{os example MFGs Tankov} for details).
Our formulation also includes a particular case of the model studied in \cite{carmona.delarue.lacker.2017.timing} (see Example \ref{os example MFGs timing}, below).

\subsection{Formulation of the model} 
Let $(\Omega, \mathcal F, \mathbb F, \mathbb P)$ be a filtered probability space, satisfying the usual conditions. 
For $0<T<\infty$, let $\mathcal{T}$ denote the set of $\mathbb{F}$-stopping times satisfying $\tau \leq T$ $\mathbb{P}$-a.s. 
Let $Z=(Z)_{t \in [0, T]}$ and $B=(B)_{t \in [0, T]}$ be progressively measurable stochastic processes, taking values in $\R^{d_1}$ and $\R^{d_2}$ for $d_1, d_2 \in \N$, respectively. Set $X:=(Z,B)$ and $d:=d_1 + d_2$.
Assume that the process $B$ has independent increments, and denote by $\mathbb F^{\text{\tiny{$B$}}}$ the right-continuous extension of the filtration generated by $B$, augmented by the $\P$-null sets. 
The process $B$ represents a common noise, and it can also be deterministic (in this case, $\mathbb F^{\text{\tiny{$B$}}}$ is the trivial filtration).
For any $t \in [0,T]$, denote by $\mathcal{F}_{t,T}^{\text{\tiny{$B$}}}$ the $\sigma$-field generated by the family of increments $ \{ B_{s_2}^i-B_{s_1}^i \, | \, t \leq s_1\leq s_1 \leq T, \, i=1,...,d_2 \}$. 
We assume that, for any $t \in [0,T]$, the $\sigma$-fields $\mathcal{F}_t$ and $\mathcal{F}_{t,T}^{\text{\tiny{$B$}}}$ are independent.

Denote by $L^{\text{\tiny{pr.}}}(\Omega \times [0,T];\mathcal{M}_{\leq 1}(\R))$ the set
of all processes taking values in the set of sub-probability measures $\mathcal{M}_{\leq 1}(\R)$, which are $\mathbb F^{\text{\tiny{$B$}}}$-progressively measurable.\
Consider two measurable functions 
\begin{align*}
    f\colon [0,T] \times \mathbb R^d \times \mathcal{M}_{\leq 1}(\R) \to \mathbb R \quad\text{and}\quad 
    g \colon [0,T] \times \mathbb R^d \to \mathbb R.
\end{align*} 
Next, for $m \in L^{\text{\tiny{pr.}}}(\Omega \times [0,T];\mathcal{M}_{\leq 1}(\R))$, we define the profit functional
\begin{equation}\label{os eq cost functional}
J(\tau,m):=\mathbb{E}\bigg[ \int_0^{\tau} f(t,X_t,m_t) dt + g(\tau , X_{\tau }) \bigg], \quad \tau \in \mathcal{T},
\end{equation}  
and consider the optimal stopping problem, parametrized by $m$, which consists of maximizing the profit functional $J(\cdot,m)$.  
For a process $m$, we say that the stopping time $\tau^m$ is optimal for $m$ if $\tau^m \in \argmax_{\mathcal{T}} J(\cdot, m) $.  

We next consider a continuous function $\psi\colon \R^d \to \R$ and the following notion of solution.
\begin{definition}\label{os MFG equilibrium}
A process $m \in L^{\text{\tiny{pr.}}}(\Omega \times[0,T];\mathcal{M}_{\leq 1}(\R))$ is a MFG equilibrium if 
$$ 
m_t(A)= \mathbb{P} [\psi(X_t) \in A, t < \tau^m | \mathcal{F}_t^{\text{\tiny{$B$}}}], \quad \text{for all } A \in \mathcal{B}(\mathbb{R}), \ t \in [0,T], \ \P \text{-a.s.,}
$$ 
for some $\tau^m \in \argmax_{\mathcal{T}} J(\cdot, m)$. 
\end{definition}

\subsection{Reformulation and preliminary results} 
In order to prove existence and approximation of the equilibria of the MFG, we embed the problem in terms of the general formulation of Section \ref{section general case}. 

Consider the set $E:=\mathcal{T}$, endowed with the lattice structure $\land, \lor$ arising from the order relation $\leq$ given by the $\P$-a.s.\ pointwise order ($\tau \leq \bar  \tau$ if and only if $\tau \leq  \bar{\tau}$ $\mathbb{P}$-a.s.). 
The lattice $E$ is complete, so that it is compact in the interval topology, see Appendix \ref{append.lattice}.
The lattice structure on $E$ allows us to directly use some of the results in \cite{Vives90} (see, in particular, Remark \ref{remark stopping Topkis}, below).  

For $\tau \in E$, we define the $\mathcal M _{\leq 1} (\R)$-valued process $p\tau$ by setting, $\P$-a.s.,
\begin{equation}\label{eq optimal stopp projection}
 (p\tau)_t(A) := \mathbb{P}[\psi(X_t) \in A, t < \tau | \mathcal{F}_t^{\text{\tiny{$B$}}} ], \quad \text{for all $A \in \mathcal{B}(\mathbb{R})$ and $t\in [0,T].$}
\end{equation}
Note that $(p\tau)_t(y,\infty) \leq \mathbb{P}[\psi(X_t) >y \,| \mathcal{F}_t^{\text{\tiny{$B$}}}]=: \mu_t^{\psi}(y,\infty)$ $\P \text{-a.s.}$, for $y\in \mathbb R$, so that 
\begin{equation}\label{eq optimal stopp projection in interval}
    (p\tau)_t \leq_{\text{\tiny{s.t.}}} \mu_t^{\psi},\; \P\text{-a.s.},\quad \text{for each $t\in [0,T]$.}
\end{equation} 
For $ m, \, \bar{m} \in  L^{\text{\tiny{pr.}}}(\Omega \times [0,T]; \mathcal{M}_{\leq 1}(\mathbb{R}))$, we define the order relation
$$
m \leq^{\text{\tiny{$L$}}} \bar{m} \iff m_t \leq^{\text{\tiny{s.t.}}} \bar{m}_t \ \mathbb P\text{-a.s., for $dt$-a.a. $t\in [0,T]$}, 
$$ 
and introduce the set of feasible distributions as
$$
L:= \{ m \in L^{\text{\tiny{pr.}}}(\Omega \times [0,T]; \mathcal{M}_{\leq 1}(\mathbb{R})) \, | \,  m \leq^{\text{\tiny{$L$}}} \mu^{\psi}  \}, 
$$
endowed with the order relation $\leq^{\text{\tiny{$L$}}}$. 
Thanks to the results in Section \ref{section lattices}, the lattice $(L,\leq^{\text{\tiny{$L$}}})$ is complete and Dedekind super complete (see in particular Example \ref{ex.lattices}).

Observe that, from the definition of $p$, we have the following monotonicity properties:
\begin{equation}\label{os eq projection omeomorfism}
p(\tau \land \bar{\tau} ) \leq^{\text{\tiny{$L$}}} p\tau \land^{\text{\tiny{$L$}}} p \bar{\tau} \leq^{\text{\tiny{$L$}}} p\tau \lor^{\text{\tiny{$L$}}} p\bar{\tau} \leq^{\text{\tiny{$L$}}} p(\tau \lor \bar{\tau}),\quad  \text{for each } \tau,\bar{\tau} \in E.
\end{equation}

The following assumption will ensure that the projection $p$ takes values in $L$ and will give the necessary integrability of the payoffs in order to gain the continuity of the functional $J$. 
\begin{assumption}\label{assumption stopping time existence} \
\begin{enumerate}
    \item\label{os continuity assumption} The processes $Z$ and $B$ are continuous;
    \item The functions $f, \, g$ are nonnegative, $g$ is continuous and 
    $$ 
    \E \bigg[ \sup_{t \in [0,T]} \Big( f(t,X_t,\mu_t^{\psi}) + g(t,X_t)\Big) \bigg] < \infty.
    $$ 
\end{enumerate}  
\end{assumption} 
 
\begin{lemma}\label{os lemma projection well defined}
The map $p\colon E \to L$ as in \eqref{eq optimal stopp projection} is well defined; i.e., $p\tau \in L$ for any $\tau \in E$.
\end{lemma}
\begin{proof}
Take $\tau \in E$.
In light of \eqref{eq optimal stopp projection in interval}, we only need to prove that the process $p\tau$ is $\mathbb{F}^{\text{\tiny{$B$}}}$-progressively measurable. 

We first show that, for each $\tau \in \mathcal T$ and $t \in [0,T]$, $\P$-a.s.,\ we have
\begin{equation}\label{eq optimal stopp conditional expectatio T}
(p\tau)_t(A)=\mathbb{P} [\psi(X_t) \in A, t < \tau | \mathcal{F}_t^{\text{\tiny{$B$}}}]=\mathbb{P} [\psi(X_t) \in A, t < \tau | \mathcal{F}_T^{\text{\tiny{$B$}}}], \quad \text{for all } A \in \mathcal{B}(\mathbb{R}). 
\end{equation}
This can be shown similarly to Remark 1 in \cite{tchuendom}. 
Indeed, for any $t \in [0,T]$, the $\sigma$-fields $\mathcal{F}_t$ and $\mathcal{F}_{t,T}^{\text{\tiny{$B$}}}$ are independent, so that the r.v.'s $Y_t^A:=\mathds{1}_{ \{ \psi(X_t) \in A \} } \mathds{1}_{\{ t< \tau \} }, \ A \in \mathcal{B}(\mathbb{R}),$ are independent from $\mathcal{F}_{t,T}^{\text{\tiny{$B$}}}$. 
Also, by assumption the $\sigma$-fields $\mathcal{F}_t^{\text{\tiny{$B$}}}$ and $\mathcal{F}_{t,T}^{\text{\tiny{$B$}}}$ are independent. 
It thus follows that, for any $ A \in \mathcal{B}(\mathbb{R})$, one has
$$
\mathbb{P} [\psi(X_t) \in A, t < \tau | \mathcal{F}_T^{\text{\tiny{$B$}}}] = \E [Y_t^A |\mathcal{F}_t^{\text{\tiny{$B$}}} \lor \mathcal{F}_{t,T}^{\text{\tiny{$B$}}} ] = \E [Y_t^A |\mathcal{F}_t^{\text{\tiny{$B$}}}] = \mathbb{P} [\psi(X_t) \in A, t < \tau | \mathcal{F}_t^{\text{\tiny{$B$}}}]
, \ \P \text{-a.s.},  
$$
which proves \eqref{eq optimal stopp conditional expectatio T}.

We can now prove that the process $p\tau$ is right-continuous $\mathbb P$-a.s. 
Indeed, for $\phi \in \mathcal C _b (\mathbb R)$, $t \in [0,T]$, and a sequence $(s_n)_n \subset [0,T]$ converging to $t$ with $s_n \geq t$, we have 
\begin{align*}
\lim_n \int_\R \phi(y) (p\tau)_{s_n}(dy)  
& = \lim_n \E [ \phi(\psi(X_{s_n})) \mathds{1}_{ \{s_n < \tau \} } | \mathcal{F}_T^{\text{\tiny{$B$}}}] \\
& =\E [ \phi(\psi(X_{t})) \mathds{1}_{ \{t < \tau \} } | \mathcal{F}_T^{\text{\tiny{$B$}}}] \\
& = \int_\R \phi(y) (p\tau)_t(dy), \quad \P\text{-a.s.},
\end{align*}
where the convergence follows by the dominated convergence theorem for conditional expectations, and using the right-continuity of $(\phi(\psi(X_{s})) \mathds{1}_{ \{s < \tau \} })_{s\in [0,T] }$ deriving from Assumption \ref{assumption stopping time existence}. 
Therefore,  $(p\tau)_{s_n}$ weakly converges to $(p\tau)_t$, $\P$-a.s., as $n\to \infty$, proving the right-continuity of $p\tau$.

Finally, since the process $p\tau$ is $\mathbb F^{\text{\tiny{$B$}}}$-adapted and right-continuous, it is $\mathbb F^{\text{\tiny{$B$}}}$-progressively measurable, completing the proof of the lemma. 
\end{proof}

\subsection{Existence and approximation of equilibria}
We enforce the following structural condition:
\begin{assumption}\label{os ass supermodularity} For each $(t, x) \in [0,T] \times \mathbb{R}^d$, the function $f(t,x,\cdot)$ is increasing; i.e.,
$f(t,x,{m}) \leq f(t,x,\bar{m})$ for any $m, \bar{m} \in \mathcal{M}_{\leq 1}( \mathbb{R})$ with $m \leq^{\text{\tiny{s.t.}}} \bar{m}$. 
\end{assumption}
From Assumption \ref{os ass supermodularity}, for $m, \bar{m} \in L$ with $m \leq^{\text{$\tiny{L}$}} \bar{m}$ and $\tau, \bar{\tau} \in E$ we have 
\begin{align*}\label{os eq J supermod proof}
J(\bar{\tau}, \bar{m})  - J(\bar{\tau} \land \tau, \bar{m})  &= \mathbb{E} \bigg[ \int_{\bar{\tau}\land \tau}^{\bar{\tau}} f(t,X_t, \bar{m}_t)dt + g(\bar{\tau } , X_{\bar{\tau} }) - g(\bar{\tau } \land \tau , X_{\bar{\tau} \land \tau }) \bigg] \\ \notag  
&\geq \mathbb{E} \bigg[ \int_{\bar{\tau}\land \tau}^{\bar{\tau}} f(t,X_t, {m}_t)dt + g(\bar{\tau }, X_{\bar{\tau} }) - g(\bar{\tau } \land \tau , X_{\bar{\tau} \land \tau })\bigg] \\ \notag
& = \mathbb{E} \bigg[ \int_{  \tau}^{{\tau}\lor \bar{\tau}} f(t,X_t, {m}_t)dt + g({\tau}\lor \bar{\tau} , X_{{\tau}\lor \bar{\tau}}) - g(\tau , X_{\tau }) \bigg] \\ \notag
&= J({\tau} \lor \bar{\tau}, {m}) - J({\tau} , {m}), 
\end{align*}
which reads as
\begin{equation}\label{os eq J supermod}
J(\bar{\tau}, \bar{m})  - J(\bar{\tau} \land \tau, \bar{m}) 
\geq J(\bar{\tau}, {m})  - J(\bar{\tau} \land \tau, {m})
= J({\tau} \lor \bar{\tau}, {m}) - J({\tau} , {m}).
\end{equation}
\begin{remark}
It is worth observing that the first inequality in \eqref{os eq J supermod} corresponds to the fact that the functional $J\colon E \times L \to \R$ has increasing differences, while the second equality in  \eqref{os eq J supermod} implies that the functionals $J(\cdot,m)\colon L \to \R$, $m \in L$, are supermodular. 
In this case, the game is said to be supermodular,  and we refer to \cite{Vives90} for further details. 
\end{remark}

We consider the best-response-maps 
\begin{equation}\label{eq stopping best response map definitions}
\hat R (m) := \begin{matrix} \argmin_E J(\cdot,m) \end{matrix} \subset E, \quad R(m):= p(\hat R (m)) \subset L, \quad m \in L. 
\end{equation}
Combining Assumption \ref{assumption stopping time existence} and Assumption \ref{os ass supermodularity} together with the definition of $L$, we obtain that, for any $m\in L$,
\begin{equation}\label{os uniform integrability costs}
\E \bigg[ \sup_{t \in [0,T]} \Big( f(t,X_t,m_t) + g(t,X_t)\Big) \bigg] \leq \E \bigg[ \sup_{t \in [0,T]} \Big( f(t,X_t,\mu_t^{\psi}) + g(t,X_t)\Big) \bigg] < \infty.
\end{equation}
This estimate, together with Assumption \ref{assumption stopping time existence}, allows to show that the functionals $J(\cdot,m): L \to \R$, $m \in L$, are continuous in the interval topology on $E$.
Therefore, arguing as in Lemma 3.1 in \cite{Vives90}, for any $m \in L$ the set $\hat R(m)$ is nonempty so that, thanks to Lemma \ref{os lemma projection well defined}, the best reply map $R\colon L \to 2^L$ is well-defined.
Moreover, $m \in L$ is an MFG equilibrium if and only if $m \in R(m)$.

\begin{remark}\label{remark stopping Topkis}
We observe that, even if the Condition \ref{assumption.bestresponsemap.nonempty.continouty projection} in Assumption \ref{assumption.bestresponsemap.nonempty} is not satisfied, the same conclusions as in Lemma \ref{lemma.bestresponse} can be deduced as follows. 
Thanks to the lattice structure on $E$ and to the supermodularity property in \eqref{os eq J supermod}, we can employ  Lemma 3.1 in \cite{Vives90}, in order to obtain that:
\begin{enumerate}
\item\label{os directed} The set $\hat R(m)$ is a lattice, i.e. for every $\tau_1,\tau_2 \in \hat R(m)$, one has $ \tau_1 \wedge \tau_2,\, \tau_1 \vee \tau_2 \in \hat R(m)$;
 \item\label{os bestresponse increasing} For all $m,\bar{m}\in L$ with $m \leq \bar m$, $ \inf_E \hat R(m) \leq \inf_E \hat R(\bar m )$ and $\sup_E \hat R(m) \leq \sup_E \hat  R(\bar m)$;
 \item\label{os bestresponsecomplete} For every $m \in L$, $\inf_E \hat R(m) \in \hat R(m)$ and $\sup_E \hat R(m)\in \hat R(m)$.
\end{enumerate}
Therefore, due to the monotonicity of the projection $p$ (see \eqref{os eq projection omeomorfism}),  for any $m \in L$, we have
\begin{equation}\label{os R hat R}
\inf R(m) = p(\begin{matrix} \inf_E \hat R(m) \end{matrix} ) \in R(m) \quad \text{and} \quad \sup R(m) = p(\begin{matrix} \sup_E \hat R(m)  \end{matrix})\in R(m), 
\end{equation}
so that the assertions of Lemma \ref{lemma.bestresponse} hold.
\end{remark}

Now, we provide the main result of this section.
We  underline that some of the conditions in Assumption \ref{assumption.general.approximation} are not satisfied (in particular, the continuity-like property of $p$ for monotone sequences of stopping times is not satisfied).
As a consequence, we obtain a result which is less general than that of Theorem \ref{theorem.main.general} (see Remark \ref{os remark learning procedure}, below), and some of the arguments in the proof of that theorem need to be adapted in order to prove existence and approximation of MFG sulutions.
 
\begin{theorem} 
The set of MFG equilibria $M$ is nonempty with $\inf M \in M$ and $\sup M\in M$. 
Moreover, if $f$ is continuous in $m$, we have that the learning procedure ${m}^n$ defined inductively by ${m}^0=\inf L$ and ${m}^{n+1}=\inf R (m^n)$ is nondecreasing in $L$ and it converges to the minimum MFG solution.
\end{theorem}
\begin{proof}
The existence and the lattice structure of equilibria follows by Tarski's fixed point theorem, since the maps $\inf R$ and $\sup R$ are nondecreasing, see Remark \ref{remark stopping Topkis}. 

We prove the convergence of the learning procedure $(m^n)_n$. 
Setting, for $n\geq 1$, $\tau_n := \inf_E \hat R(m^{n-1})$, by Remark \ref{remark stopping Topkis}, we have that $\tau_n \leq \tau_{n+1}$, $m^n \leq^{\text{$\tiny{L}$}} m^{n+1}$, and $m^n = p\tau_n$ for any $n\geq 1$.
By the completeness of the lattices $E$ and $L$, we can define $\tau_*:= \sup_E \{ \tau_n | n \geq 1 \}$ and $m^* := \sup_n m^n$, and  we have
\begin{equation}\label{os eq convergence to the sup}
    \tau_n \to \tau_* \ \P\text{-a.s.} \quad \text{and} \quad m_t^n \to m_t^* \ \text{weakly} \ \P \otimes dt \text{-a.e.}, \quad \text{as $n\to \infty$.} 
\end{equation}
By definition of $m^n$ and $\tau_n$, for any $n \geq 1$, we have $J(\tau_n, m^{n-1}) \geq J(\tau, m^{n-1})$ for any $\tau \in E$. 
Therefore, taking limits as $n \to \infty$ (justified by the integrability in \eqref{os uniform integrability costs} and the convergence in \eqref{os eq convergence to the sup}), we obtain $J(\tau_*, m^*) \geq J(\tau, m^*)$ for any $\tau \in E$, so that \begin{equation}\label{eq stopp tau star best response}
    \tau_* \in \hat{R}(m^*).
\end{equation}

Moreover, the sequence $(\tau_n)_n$ increasingly converges to $\tau_*$, $\mathbb P$-a.s.,\ as $n \to \infty$. 
Therefore, using the dominated convergence theorem for conditional expectations and exploiting the left-continuity of the map $\mathds{1}_{\{t < \cdot \} }$, we find that, $\P$-a.s.,
\begin{align*}
(p\tau_*)_t(y) & = \mathbb E [ \mathds{1}_{\{\psi(X_t) >y\}} \mathds{1}_{\{ t<\tau_*\}}| \mathcal{F}_T^{\text{\tiny{$B$}}}]\\
&= \lim_n \mathbb E [ \mathds{1}_{\{\psi(X_t) >y\}} \mathds{1}_{\{ t<\tau_n\}}| \mathcal{F}_T^{\text{\tiny{$B$}}}]\\
&= \lim_n (p\tau_n)_t(y) = \lim_n m^n_t(y), \quad \text{for any $(t, y)\in [0,T] \times \R$.}
\end{align*} 
The latter, thanks to the convergence in \eqref{os eq convergence to the sup}, in turn implies that $p\tau_* = \sup_n m^n = m^*$. 
This, together with \eqref{eq stopp tau star best response}, gives that $m^* \in R(m^*)$, so that $m^*$ is a MFG solution. 

The fact that $m^*$ is the minimal MFG solution follows as in the proof of the general Theorem \ref{theorem.main.general}, and this completes the proof of the theorem.  
\end{proof}

\subsection{Comments and examples} 
\begin{remark}\label{os remark learning procedure}
We point out that, since the function $\mathds{1}_{\{t < \,\cdot\, \} }$ is not right-continuous, the learning procedure $(m^n)_n \subset L$, which is defined inductively by $m^0 := \sup L$ and $m^{n+1}:= \sup R (m^n)$, cannot be shown to converge to a MFG equilibrium. 
\end{remark}

\begin{example} [MFGs of timing with common noise and interaction of scalar type]\label{os example MFGs Tankov}
As an example, we may consider a MFG in which the state variable $Z$ evolves according to the SDE 
$$ 
dZ_t = b(t,Z_t) dt + \sigma(t,Z_t) dW_t + \sigma^o(t,Z_t) dB_t, \ t \in [0,T],    
$$ 
for functions $(b,\sigma,\sigma^o)\colon [0,T] \times \R^{d_1} \to \times \R^{d_1 }  \times \R^{d_1 \times d_1} \times \R^{d_1 \times d_1}$ satisfying the usual Lipschitz conditions, and for  $\mathbb F$-adapted independent Brownian motions $W$ and $B$ taking values in $\R^{d_1}$ and $\R^{d_2}$, respectively.\
Moreover, one may consider a running profit function $f$, which enjoys a scalar nondecreasing dependence on the measure; that is, $f$ is given by $f(t,x,m):= \bar{f}(t,x, \left\langle \phi, m \right\rangle)$, where $\bar{f}(t,x,\cdot)$ is nondecreasing, $\phi\colon \R \to [0,\infty)$ is nondecreasing, and $\left\langle \phi, m \right\rangle:= \int_{[0,\infty)} \phi(x)\, dm(x)$. 

Such a setting resembles the one considered in \cite{bouveret.dumitrescu.tankov.20}, even if several differences arise between the problem in  \cite{bouveret.dumitrescu.tankov.20} and ours. 
Firstly, in \cite{bouveret.dumitrescu.tankov.20} no common noise is considered, and a nondegeneracy condition on the volatility matrix is needed in order to employ results from PDE theory. These requirements are not needed for our lattice-theoretic approach to work. 
Secondly, in \cite{bouveret.dumitrescu.tankov.20} -- in order to establish uniqueness of the MFG equilibrium -- a suitable anti-monotonicity property is imposed on the dependence of the running profit function with respect to the measure variable (see Assumption 8 therein), whereas, in the setting of this example, we need that the function $f$ is nondecreasing with respect to its third argument. 
Thirdly, a convergence result is established in \cite{bouveret.dumitrescu.tankov.20} for potential games, while the potential structure is not needed for our learning procedure to work.
\end{example} 

\begin{example}[MFGs of timing with common noise]\label{os example MFGs timing}
A particular example is when $Z_t=(t,\bar{Z}_t)$ and $\psi(t,\bar z,b)=t$, for $(t,\bar z,b) \in \R^{d_1} \times \R^{d_2}$. 
In this case, the fixed point condition in Definition \ref{os MFG equilibrium} reduces to an identity on the space of $\mathcal P ([0,T])$-valued random variables. 
In other words, an equilibrium is an $\mathcal F_T^B$-adapted $\mathcal P ([0,T])$-valued random variable $m$ such that, $\P$-a.s.,
$$
m_t= \P [ t< \tau^m | \mathcal{F}_T^{\text{$\tiny{B}$}}], \quad \text{for any $t\in [0,T]$ and some $\tau^m \in \textstyle{\argmin _E }J(\cdot,m) $}.
$$
This example corresponds to a particular case of the MFG of timing with common noise discussed in \cite{carmona.delarue.lacker.2017.timing}. 
\end{example}

\appendix

\section{Lattice-theoretic preliminaries}\label{append.lattice}

In this section, we collect some notions and preliminaries for lattices. Throughout, we consider a fixed lattice $L$, i.e.,~a partially ordered set (poset) in which every finite nonempty subset has a least upper bound and a greatest lower bound. We start with the following definition.

\begin{definition}\label{def.dedekind}\
\begin{enumerate}
  \item[a)] We say that $L$ is \textit{Dedekind $\sigma$-complete} if every countable nonempty subset, that is bounded above or below, has a least upper bound or a greatest lower bound, respectively. We say that $L$ is \textit{Dedekind complete} if every nonempty subset, that is bounded above or below, has a least upper bound or a greatest lower bound, respectively. We say that $L$ is \textit{Dedekind super complete} if every nonempty subset, that is bounded above or below, has a countable subset with the same least upper bound or greatest lower bound, respectively. We say that $L$ is \textit{complete} if every nonempty subset of $L$ has a least upper bound and a greatest lower bound.
  \item[b)] We say that a set $M\subset L$ is \textit{directed upwards} or \textit{directed downwards} if, for all $x,y\in M$, there exists some $z\in M$ with $x\vee y\leq z$ or $x\wedge y\geq z$, respectively.
 \end{enumerate}
\end{definition}

\begin{definition}\label{def:increasing}
 We say that a map $F\colon L\to \R$ is \textit{strictly increasing} if
 \begin{enumerate}
  \item[(i)] $F(x)\leq F(y)$ for all $x,y\in L$ with $x\leq y$,
  \item[(ii)] for all $x,y\in L$ with $x\leq y$ and $F(x)=F(y)$, it follows that $x=y$.
 \end{enumerate}
\end{definition}

The following lemma is a special case of \cite[Lemma A.3]{nendel}, and gives a sufficient condition for a Dedekind $\sigma$-complete lattice to be Dedekind super complete. For the proof, we refer to \cite{nendel}.

\begin{lemma}\label{lem.blms}
   Let $L$ be a Dedekind $\sigma$-complete lattice. If there exists a strictly increasing map $F\colon L\to \R$, then $L$ is Dedekind super complete.
\end{lemma}

A fundamental result by Birkhoff \cite[Section X.12, Theorem X.20]{birkhoff1967} and Frink \cite{frink1942} is that completeness of the lattice $L$ corresponds to the compactness of $L$ w.r.t.\ the so-called interval topology, whose definition we briefly recall here.

\begin{definition}\label{def.intervaltop}
The \textit{interval topology} on $L$ is the smallest topology $\tau$ on $L$ such that all closed intervals of the form
$$(-\infty, a]:=\{x\in L \, |\, x\leq a\}\quad \text{and}\quad [a,\infty):=\{x\in L\, |\, x\geq a\}, \quad \text{for }a\in L$$
are closed w.r.t.\ $\tau$.
\end{definition}

\smallskip
\textbf{Acknowledgements.} 
Financial support by the German Research Foundation (DFG) through the Collaborative Research Centre  1283/2 2021 -- 317210226 is gratefully acknowledged.

\bibliographystyle{siam} 
\bibliography{main.bib}
\end{document}